\documentclass[preprint,reqno]{imsart}

\RequirePackage[OT1]{fontenc}
\RequirePackage{amsthm,amsmath}
\RequirePackage[numbers]{natbib}
\RequirePackage[colorlinks,citecolor=blue,urlcolor=blue]{hyperref}

\arxiv{1708.02521}
\usepackage{mathpazo}
\usepackage{bbm}
\usepackage{enumitem}
\usepackage{mathtools}
\usepackage{amsmath}
\usepackage{mathrsfs}
\usepackage{appendix}
\usepackage{appendix}
\startlocaldefs
\numberwithin{equation}{section}
\theoremstyle{plain}

\endlocaldefs
\newtheorem{theorem}{Theorem}[section]
\newtheorem{lemma}[theorem]{Lemma}
\newtheorem{remark}[theorem]{Remark}
\newtheorem{definition}[theorem]{Definition}
\newtheorem{proposition}[theorem]{Proposition}

\allowdisplaybreaks[1]
\begin{document}

\begin{frontmatter}
\title{Stein's method for multivariate Brownian approximations of sums under dependence}
\runtitle{Multivariate Brownian approximations via Stein's method}

\begin{aug}
\author{\fnms{Miko{\l}aj J.} \snm{Kasprzak}\ead[label=e1]{mikolaj.kasprzak@uni.lu}},

\runauthor{Miko{\l}aj J. Kasprzak}

\affiliation{University of Luxembourg}

\address{University of Luxembourg\\
Department of Mathematics\\
Maison du Nombre\\
6 Avenue de la Fonte\\
L-4364 Esch-sur-Alzette\\
Luxembourg\\
\printead{e1}\\
\phantom{E-mail:\ }}
\end{aug}
\begin{abstract}
We use Stein's method to obtain a bound on the distance between scaled $p$-dimensional random walks and a $p$-dimensional (correlated) Brownian motion. We consider dependence schemes including those in which the summands in scaled sums are weakly dependent and their $p$ components are strongly correlated. As an example application, we prove a functional limit theorem for exceedances in an $m$-scans process, together with a bound on the rate of convergence. We also find a bound on the rate of convergence of scaled U-statistics to Brownian motion, representing an example of a sum of strongly dependent terms. 
\end{abstract}


\begin{keyword}
\kwd{Stein's method}
\kwd{functional convergence}
\kwd{Brownian motion}
\kwd{exceedances of the scans process}
\kwd{U-statistics}
\kwd{\\~\\ © 2020, Elsevier. Licensed under the Creative Commons Attribution-NonCommercial-NoDerivatives 4.0 International http://creativecommons.org/licenses/by-nc-nd/4.0/ }
\end{keyword}
\end{frontmatter}
\section{Introduction}
In the seminal paper \cite{diffusion}, Barbour addressed the problem of providing bounds on the rate of convergence in functional limit results (or invariance principles as they are often called in the literature). He observed that the celebrated Stein's method, first introduced in \cite{stein} as a tool for proving the Central Limit Theorem, may also be used in  the setup of the \textit{Functional} Central Limit Theorem. This theorem, whose early versions are attributed to Donsker \cite{donsker}, says that for a sequence of i.i.d. real random variables $(X_n)_{n=1}^{\infty}$ with mean zero and unit variance, the random process
\begin{equation}\label{int}
\mathbf{Y}_n(t)= n^{-1/2}\sum_{i=1}^{\lfloor nt\rfloor}X_i,\quad t\in[0,1]
\end{equation}
converges in distribution to the standard Brownian motion with respect to the Skorokhod topology. 

Through a careful and technical adaptation of Stein's method to the framework of Brownian-motion approximation and a subsequent repetitive use of Taylor's theorem, Barbour \cite{diffusion} proved a powerful estimate on a distance between the law of $\mathbf{Y}_n$ in (\ref{int}) and the Wiener measure. Specifically, he considered test functions $g$ acting on the Skorokhod space $D \left( [0,1], \mathbb{{R}} \right)$ of c\`adl\`ag real-valued maps on $[0,1]$, such that $g$ takes values in the reals, does not grow faster than a cubic, is twice Fr\'echet differentiable and its second derivative is Lipschitz. Denoting by $\mathbf{Z}$ the Brownian motion on $[0,1]$ and adopting the notation of (\ref{int}), his result says that
$$|\mathbb{E}g(\mathbf{Y}_n)-\mathbb{E}g(\mathbf{Z})|\leq C_g\frac{\mathbb{E}|X_1|^3+\sqrt{\log n}}{\sqrt{n}},$$
where $C_g$ is a constant, independent of $n$, yet depending on the (carefully defined) \textit{smoothness properties} of $g$. Among the applications and extensions considered by Barbour are an analysis of the empirical distribution function of i.i.d. random variables and the Wald-Wolfowitz theorem often used to construct tests in non-parametric statistics \cite{wald1940}.

Our aim in this paper is to extend the results of \cite{diffusion} to approximations of scaled sums of univariate and \textit{multivariate} random variables with different \textit{dependence structures} by univariate and \textit{multivariate} Wiener processes.
\subsection{Motivation}
Functional limit results play an important role in applied fields. Researchers often choose to model discrete phenomena with continuous processes arising as scaling limits of discrete ones. The reason is that those scaling limits may be studied using stochastic analysis and are more robust to changes in local details. Questions about the rate of convergence in functional limit results are equivalent to ones about the error those researchers make when doing so. Obtaining bounds on a certain distance between the scaled discrete and the limiting continuous processes provides a way of quantifying this error.

Our motivation in this paper comes from the desire to fill in a gap in the theory but we are also motivated by examples related to applications.

One of those, studied in the example in Section \ref{ex_exceedances} of this paper, considers \textit{exceedances of the $m$-scans process}. For a sequence of i.i.d. random variables $X_1,X_2,\dots$, the one-dimensional $m$-scans process is given by $R_i=\sum_{k=0}^{m-1}X_{i+k}$. The number of its exceedances of a real number $a$ is given by 
$$Y=\sum_{i=1}^n\mathbb{1}[R_i>a].$$
As noted in \cite[Example 9.2]{normal_approx}, this statistic has been studied by many authors, including \cite{glaz} and \cite{naus}. It is of high importance in many areas of applied statistics and has been used, for instance, to evaluate the significance of observed inhomogeneities in the distribution of markers along the length of long DNA sequences (see \cite{dembo_karlin, karlub_brede}). $Y$ may be normalized and centralized and then shown to converge in distribution to the standard normal law. Berry-Esseen bounds on the rate of this convergence have been found in \cite[Theorem 4.1]{dembo_rinott} and \cite[Example 9.2]{normal_approx}. We are interested in studying the functional convergence of a multidimensional version of $Y$.

Another example concerns \textit{bivariate U-statistics} and is treated in Theorem \ref{theorem_u_stats} of this paper. Bivariate U-statistics are defined to be random variables of the form:
$$S_n^2(h)=\sum_{1\leq i_1<i_2\leq n} h(X_{i_1},X_{i_2}),\quad n\geq 1$$
for a symmetric real (or complex) function $h$ on $\mathcal{S}^2$ (where $\mathcal{S}$ is some measurable space) and a sequence of i.i.d. random variables $(X_i)_{i\geq 1}$ taking values in $\mathcal{S}$. Because of their appealing properties, they are central objects in the field of Mathematical Statistics, as described in \cite{encyclopedia} and many commonly used statistics can be expressed in terms of certain U-statistics or approximated by them. They also appear in decompositions of more general statistics into sums of terms of a simpler form (see, e.g. \cite[Chapter 6]{serfling} or \cite{rubin_vitale} and \cite{vitale}) and play an important role in the study of random fields (see, e.g. \cite[Chapter 4]{christofides}). The appealing properties of \textit{non-degenerate bivariate U-statistics}, i.e. those such that, for
$$w(x)=\mathbb{E}h(x,X_1),$$
$0<\text{Var}[w(X_1)]<\infty$, include their asymptotic behaviour. It can be described by a Strong Law of Large Numbers (\cite{hoeffding2}), a central limit theorem (\cite{hoeffding}) or the functional central limit theorem (e.g. \cite[Chapter XI]{hilbert}), which will be studied in this paper. Other interesting results include those connected to large deviations for U-statistics (see \cite{lowe}), Berry-Esseen-type bounds (see \citep{chen_shao1}) and other bounds on the speed of convergence in the U-statistic CLT (see \cite{rinott}). \textit{Degenerate U-statistics} have also received much attention in the recent years with \cite{dobler2017} providing bounds on the speed of convergence in de Jong's theorem \cite{de_jong} and proving its multidimensional version. 


Our theoretical motivation is expressed in Proposition \ref{local1} of this paper. It seems natural to ask whether techniques similar to those of \cite{diffusion} may be used to study a process of the form
\begin{equation}\label{int_proc}
t\mapsto n^{-1/2}\sum_{i=1}^{\lfloor nt\rfloor}X_i,\quad t\in[0,1]
\end{equation}
where $\lbrace X_i:i=1,\dots,n\rbrace$ is a collection of i.i.d. random vectors in $\mathbb{R}^p$ for $p>1$ with a given covariance matrix $\Sigma$. Interesting questions arising include those about the rate of convergence of the process in (\ref{int_proc}) to the correlated $p$-dimensional Brownian motion created from a standard Brownian motion $\mathbf{B}$ by premultiplying it by $\Sigma^{1/2}$. In this context, the role played by $\Sigma$ in the quality of this approximation seems worth paying attention to. 
\subsection{Contribution of the paper}
The main achievements of the paper are the following:
\begin{enumerate}[label=(\alph*)]
\item A very general result providing a bound on the distance between a process of the form
$$\mathbf{Y}_n(t)=\left(\sum_{i=1}^{\lambda_1}X_{i,1}J_{i,1}(t),\dots,\sum_{i=1}^{\lambda_p}X_{i,p}J_{i,p}(t)\right),\quad t\in[0,1],$$
where:
\begin{itemize}
\item the numbers $\lambda_j$ are such that $\lambda_j\leq n$;
\item $p$ is a fixed positive integer;
\item the collection of vectors $X_i=(X_{i,1},\dots,X_{i,p})$ for $i=1,\dots, n$ is allowed to be \textit{dependent} and those vectors themselves are allowed to have non-identity covariance matrices;
\item the collection of (possibly random) functions 
$$\left\lbrace J_{i,k}\in D([0,1],\mathbb{R}):i=1,\dots n, k=1,\dots, p\right\rbrace$$
is independent of the collection  of vectors $(X_i)_{i=1}^n$ from the previous point;
\end{itemize}
and a correlated $p$-dimensional Brownian motion. The bound is presented in Theorem \ref{local3} and provides a substantial extension of the result of \cite{diffusion}, which bounds the rate of convergence in the classical, one-dimensional Donsker's invariance principle.
\item A novel functional central limit theorem involving the number of exceedances in the multidimensional $m$-scans process, together with bounds on the rate of convergence, presented in the example in Section \ref{ex_exceedances}.
\item A novel bound on the rate of convergence in the functional central limit theorem for non-degenerate, bivariate U-statistics (for a classical proof of the theorem see, for instance, \cite{hall}), which is presented in Theorem \ref{theorem_u_stats}.
\item A technical result, presented in Proposition \ref{prop_m}, showing that our bounds' converging to zero implies weak convergence of the underlying processes with respect to the Skorokhod and uniform topologies. This result is a direct extension of \cite[Proposition 3.1]{functional_combinatorial} to the multidimensional setting.
\end{enumerate}
We provide explicit values for all the constants appearing in our bounds. To our best knowledge, none of the authors who have considered functional approximations with Stein's method so far has done so. We do it as we hope that this will make our results more powerful when used in applications.

The technique which is central in obtaining all the bounds is Stein's method.
\subsection{Stein's method for distributional approximation}
In \cite{stein} it is observed that a random variable $Z$ has standard normal law if and only if $\mathbb{E}Zf(Z)=\mathbb{E}f'(Z)$ for all smooth functions $f$. Therefore, if, for a random variable $W$ with mean zero and unit variance, $\mathbb{E}f'(W)-\mathbb{E}Wf(W)$ is close to zero for a large class of functions $f$, then the law of $W$ should be approximately Gaussian. This leads to a method of bounding the speed of convergence to the normal distribution. Instead of evaluating $|\mathbb{E}h(W)-\mathbb{E}h(Z)|$ directly for a given function $h$, one can first find an $f=f_h$ solving the following \textit{Stein equation}:
$$f'(w)-wf(w)=h(w)-\mathbb{E}h(Z)$$
and then find a bound on $|\mathbb{E}f'(W)-\mathbb{E}Wf(W)|$. This approach, called \textit{Stein's method}, often turns out to be surprisingly easy and has also proved to be useful for approximations by distributions other than normal.

The aim of the generalised version of Stein's method is to find a bound for the quantity $|\mathbb{E}_{\nu_n}h-\mathbb{E}_{\mu}h|$, where $\mu$ is the target (known) distribution, $\nu_n$ is the approximating law and $h$ is chosen from a suitable class of real-valued test functions $\mathcal{H}$. The procedure can be described in terms of three steps. First, an operator $\mathcal{A}$ acting on a class of real-valued functions is sought, such that $$\left(\forall f\in\text{Domain}(\mathcal{A})\quad\mathbb{E}_{\nu}\mathcal{A}f=0\right)\quad \Longleftrightarrow \quad\nu=\mu,$$
where $\mu$ is the target distribution. Then, for a given function $h\in\mathcal{H}$, the Stein equation
$$\mathcal{A}f=h-\mathbb{E}_{\mu}h$$
has to be solved. Finally, using properties of the solution and various mathematical tools (among which the most popular are Taylor's expansions in the continuous case, Malliavin calculus, as described in \cite{nourdin}, and coupling methods), an explicit bound is sought for the quantity $|\mathbb{E}_{\nu_n}\mathcal{A}f_h|$.

An accessible account of the method can be found, for example, in the surveys \cite{reinert} and \cite{ross} as well as the books \cite{janson} and \cite{normal_approx}, which treat the cases of Poisson and normal approximation, respectively, in detail. The reference \cite{swan} is a database of information and publications connected to Stein's method.

Approximations by laws of diffusion processes have not been covered in the Stein's method literature very widely, with the notable exceptions of \cite{diffusion, functional_combinatorial, shih, Coutin} and recently \cite{decreusefond2, kasprzak1, kasprzak3}. Our aim in this paper is to develop it in a direction not previously explored by other authors while completely natural given the direction in which the finite-dimensional Stein's method literature has evolved.

\subsection{Structure of the paper}
In Section \ref{section 2} we define the spaces of test functions we will be working with and the corresponding norms which will appear in the bounds. We also present Proposition \ref{prop_m} giving circumstances under which the bounds obtained later in the paper converging to zero imply weak convergence of the considered probability distributions. Section \ref{section3} gives statements of the main results of the paper, mentioned above. Section \ref{ex_exceedances} presents the example concerning exceedances of an $m$-scans process. Section \ref{section5} contains all the proofs preceded by finding the Stein equation for approximation by the law of interest, solving it and examining properties of the solutions. In the appendix we present the proof of the aforementioned Proposition \ref{prop_m}.

\section{Notation and spaces $M$, $M^1$, $M^2$ and $M^0$}\label{section 2}
The following notation is used throughout the paper. For a function $w$ defined on the interval $[0,1]$ and taking values in a Euclidean space, we define 
$$\|w\|=\sup_{t\in[0,1]}|w(t)|,$$ 
where $|\cdot|$ denotes the Euclidean norm. We also let $p$ be an integer such that $p\geq 1$ and $D^p=D([0,1],\mathbb{R}^p)$ be the Skorokhod space of all c\`adl\`ag functions on $[0,1]$ taking values in $\mathbb{R}^p$. In the literature, this space is usually equipped with the \textit{Skorokhod topology} generated by the \textit{Skorokhod metric} $\sigma$ given by
$$\sigma(w,v)=\inf_{\lambda\in\Lambda}\max\lbrace\|\lambda-I\|,\|w-v\circ\lambda\|\rbrace,$$
where $I$ is the identity function and $\Lambda$ is the set of all strictly increasing continuous bijections on $[0,1]$. We will most often consider the topology generated by the supremum norm, though. 

In the sequel, for $i=1,\dots,p$, $e_i$ will denote the $i$th unit vector of the canonical basis of $\mathbb{R}^p$ and the $i$th component of $x\in\mathbb{R}^p$ will be represented by $x^{(i)}$, i.e. $x=\left(x^{(1)},\dots,x^{(p)}\right)$.

Let $p\in\mathbb{N}$. Let us define:
$$\|f\|_L:=\sup_{w\in D^p}\frac{|f(w)|}{1+\|w\|^3}\text{,}$$
and let $L$ be the Banach space of continuous functions $f:D^p\to\mathbb{R}$ such that $\|f\|_L<\infty$. Following \cite{diffusion}, we now define $M\subset L$ to be the set of the twice Fr\'echet differentiable functions $f$, such that:
\begin{equation}\label{space_m}
\|D^2f(w+h)-D^2f(w)\|\leq k_f\|h\|\text{,}
\end{equation}
for some constant $k_f$, uniformly in $w,h\in D^p$. By $D^kf$ we mean the $k$-th Fr\'echet derivative of $f$ and the norm of $k$-linear form $B$ on $L$ is defined to be 
$$\|B\|=\sup_{\|h_i\|\leq 1\,\forall i=1,\dots,k} |B[h_1,...,h_k]|,$$ 
where $$B[h_1,...,h_k]$$ denotes $B$ applied to arguments $h_1,\dots,h_k$. Note the following lemma, which can be proved in an analogous way to that used to show (2.6) and (2.7) of \cite{diffusion}. We omit the proof here. 
\begin{lemma}\label{first_der}
For every $g\in M$, let:
\begin{align*}
\|g\|_M:=&\sup_{w\in D^p}\frac{|g(w)|}{1+\|w\|^3}+\sup_{w\in D^p}\frac{\|Dg(w)\|}{1+\|w\|^2}+\sup_{w\in D^p}\frac{\|D^2g(w)\|}{1+\|w\|}\\
&+\sup_{w,h\in D^p}\frac{\|D^2g(w+h)-D^2g(w)\|}{\|h\|}.
\end{align*}
Then, for all $g\in M$, we have $\|g\|_M<\infty$.
\end{lemma}
For future reference, we let $ M^1\subset M$ be the class of functionals $g\in M$ such that:
\begin{align}
\|g\|_{M^1}:=&\sup_{w\in D^p}\frac{|g(w)|}{1+\|w\|^3}+\sup_{w\in D^p}\|Dg(w)\|+\sup_{w\in D^p}\|D^2g(w)\|\nonumber\\
&+\sup_{w,h\in D^p}\frac{\|D^2g(w+h)-D^2g(w)\|}{\|h\|}<\infty\text{.}\label{m_1}
\end{align}
and $ M^2\subset M$ be the class of functionals $g\in M$ such that:
\begin{align}
\|g\|_{M^2}:=&\sup_{w\in D^p}\frac{|g(w)|}{1+\|w\|^3}+\sup_{w\in D^p}\frac{\|Dg(w)\|}{1+\|w\|}+\sup_{w\in D^p}\frac{\|D^2g(w)\|}{1+\|w\|}\nonumber\\
&+\sup_{w,h\in D^p}\frac{\|D^2g(w+h)-D^2g(w)\|}{\|h\|}<\infty\text{.}\label{m_2}
\end{align}
We also let $M^0$ be the class of functionals $g\in M$ such that:
\begin{align}
\|g\|_{M^0}:=&\sup_{w\in D^p}|g(w)|+\sup_{w\in D^p}\|Dg(w)\|+\sup_{w\in D^p}\|D^2g(w)\|\nonumber\\
&+\sup_{w,h\in D^p}\frac{\|D^2g(w+h)-D^2g(w)\|}{\|h\|}<\infty\text{.}\nonumber
\end{align}
We note that $M^0\subset M^1\subset M^2\subset M$.  We shall refer to those different classes of functions in the results presented in the remainder of this paper. In each case we aim to obtain our bounds for the largest possible class, yet it is not always possible to do so for class $M$ or even $M^2$. Hence, the introduction of the above presented restrictions of $M$ is necessary for a recovery of the full strength of our results.

The next proposition is a $p$-dimensional version of \cite[Proposition 3.1]{functional_combinatorial} and shows conditions, under which convergence of the sequence of expectations of a functional $g$ under the approximating measures to the expectation of $g$ under the target measure for all $g\in M^0$ implies weak convergence of the measures of interest. The proposition will be later used to conclude weak convergence from bounds derived in the theorems of the next section. Its proof can be found in the Appendix.
\begin{definition}
$Y\in D\left([0,1],\mathbb{R}^p\right)$ is called piecewise constant if $[0,1]$ can be divided into intervals of constancy $[a_k,a_{k+1})$ such that $(Y(t_1)-Y(t_2))=0$ for all $t_1,t_2\in[a_k,a_{k+1})$.
\end{definition}

\begin{proposition}\label{prop_m}
Suppose that, for each $n\geq 1$, the random element $\mathbf{Y}_n$ of $D^p$ is piecewise constant and let $r_n>0$ be such that the intervals of constancy are of length at least $r_n$. Let $\left(\mathbf{Z}_n\right)_{n\geq 1}$ be random elements of $D^p$ converging in distribution in $D^p$, with respect to the Skorokhod topology, to a random element $\mathbf{Z}\in C\left([0,1],\mathbb{R}^p\right)\subset D^p$. If there exists a sequence $(\kappa_n)_{n\geq 1}$ such that $\kappa_n\log^2(1/r_n)\xrightarrow{n\to\infty}0$ and
\begin{equation}\label{assumption}
|\mathbb{E}g(\mathbf{Y}_n)-\mathbb{E}g(\mathbf{Z}_n)|\leq C\kappa_n\|g\|_{M^0}
\end{equation}
for each $g\in M^0$ then $\mathbf{Y}_n\Rightarrow \mathbf{Z}$ (converges weakly) in $D^p$, in both the uniform and the Skorokhod topology.
\end{proposition}
\begin{remark}
The formulation of Proposition \ref{prop_m} is almost identical to that of \cite[Proposition 3.1]{functional_combinatorial} with the only difference being that $\mathbf{Y}_n$ and $\mathbf{Z}_n$ are allowed to be $p$-dimensional for $p>1$. For completeness, the appendix contains a more detailed proof than the one presented in \cite{functional_combinatorial}, which may be used by the reader to derive extensions or other versions of the result.
\end{remark}

\section{Main results}\label{section3}
\subsection{Scaled sum of dependent vectors with dependent components}
Theorem \ref{local3} below studies a scaled sum of locally dependent terms whose components are (strongly) dependent. It bounds the error on its approximation by a correlated Brownian motion for test functions in $M^1$.
\begin{theorem}[Dependent components and locally dependent summands]\label{local3}
Let $n$ and $p$ be positive integers. Consider an array of mean-zero random variables 
$$\lbrace X_{i,j}:i=1,...,n,j=1,...,p\rbrace,$$ 
with a positive definite covariance matrix $\tilde{\Sigma}_n$. Let
\begin{enumerate}[label=(\alph*)]
\item $\lambda_j\leq n$, for $j=1,\dots p$, be deterministic positive integers;
\item $\mathbb{A}_i\subset\lbrace 1,2,...,n\rbrace$, for $i=1,\dots n$ be a set such that $X_i=\left(X_{i,1},\dots,X_{i,p}\right)$ is independent of $\left\lbrace X_{j}:j\in \mathbb{A}_i^c\right\rbrace$;
\item $\mathbb{A}_{ij}\subset\lbrace 1,\dots,n\rbrace$, for $i,j=1,\dots, n$ be a set such that $(X_i,X_j)$ and $\lbrace X_k:k\in A_{ij}^c\rbrace$ are independent.
\item  $J_{i,k}\in D\left([0,1],\mathbb{R}\right)$ for $i=1,\dots, n$ and  $k=1,\dots, p$, be (possibly random) functions, independent of the family $\lbrace X_{i,k}:\,i=1,\dots,n,\,k=1,\dots,p\rbrace$.
\end{enumerate}
Assume that:
\begin{equation*}
\underset{k_1,k_2,k_3\in\lbrace 1,\dots,p\rbrace}{\sup_{i_1,i_2,i_3\in\lbrace 1,\dots, n\rbrace}}\mathbb{E}\left[\|J_{i_1,k_1}\|\|J_{i_2,k_2}\|\|J_{i_3,k_3}\|\right]<\infty.
\end{equation*}
Let
$$\mathbf{Y}_n(t)=\left(\sum_{i=1}^{\lambda_1}X_{i,1}J_{i,1}(t),\dots,\sum_{i=1}^{\lambda_p}X_{i,p}J_{i,p}(t)\right),\quad t\in[0,1].$$
Furthermore, for a standard $p$-dimensional Brownian motion $\mathbf{B}$ and a positive definite covariance matrix $\Sigma\in\mathbb{R}^{p\times p}$, let $\mathbf{Z}=\Sigma^{1/2}\mathbf{B}
$.
Then, for any $g\in M^1$, as defined by (\ref{m_1}):
$$|\mathbb{E}g(\mathbf{\mathbf{Y}}_n)-\mathbb{E}g(\mathbf{Z})|\leq\|g\|_{M^1} \sum_{i=1}^7 \epsilon_i,$$
where:

\begin{align*}
&\epsilon_1=\frac{1}{6}\sum_{i=1}^n\mathbb{E}\left\lbrace\left(\sum_{k,l,m=1}^p\left[\left(X_{i,k}\right)^2\|J_{i,k}\|^2\mathbb{1}_{[1,\lambda_k]}(i)\left(\sum_{j\in \mathbb{A}_i}X_{j,l}\|J_{j,l}\|\mathbb{1}_{[1,\lambda_l]}(j)\right)^2\right.\right.\right.\nonumber\\
&\left.\left.\left.\phantom{\frac{\|g\|_{M}}{3}\sum_{i=1}^n\mathbb{E}\lbrace}\cdot\left(\sum_{j\in \mathbb{A}_i}X_{j,m}\|J_{j,m}\|\mathbb{1}_{[1,\lambda_m]}(j)\right)^2\right]\right)^{1/2}\right\rbrace;\\
&\epsilon_2=\frac{1}{3}\sum_{i=1}^n\sum_{j\in \mathbb{A}_i}\sum_{k,l=1}^p\mathbb{E}\left\lbrace\vphantom{\left[\left(\sum_r\in A_j\right)^2\right]^{1/2}}\left[\vphantom{\left(\sum_r\in A_j\right)^2}\sum_{m=1}^p\left(\vphantom{\sum_r\in A_j}X_{i,k}\,\|J_{i,k}\|\,X_{j,l}\,\|J_{j,l}\|\,\mathbb{1}_{[1,\lambda_k]}(i)\mathbb{1}_{[1,\lambda_l]}(j)\right.\right.\right.\nonumber\\
&\left.\left.\left.\phantom{\frac{\|g\|_M}{3}\sum_{i=1}^n\sum_{j\in \mathbb{A}_i}\sum_{k,l=1}^p\mathbb{E}\lbrace}\cdot\sum_{r\in \mathbb{A}_{ij}\,\cap\, \mathbb{A}_i^c}X_{r,m}\|J_{r,m}\|\mathbb{1}_{[1,\lambda_m]}(r)\right)^2\right]^{1/2}\right\rbrace;\\
&\epsilon_3=\frac{1}{3}\sum_{i=1}^n\sum_{j\in\mathbb{A}_i}\sum_{k,l=1}^p\left\lbrace\vphantom{\left[\sqrt{\left(\sum_1^p\right)^2}\right]}\left|\mathbb{E}\left[X_{i,k}X_{j,l}\right]\right|\mathbb{1}_{[1,\lambda_k]}(i)\mathbb{1}_{[1,\lambda_l]}(j)\right.\nonumber\\
&\phantom{\frac{\|g\|_{M}}{3}\sum_{i=1}^n\sum_{j\in\mathbb{A}_i}\sum_{k,l=1}^p\lbrace}\cdot\left.\mathbb{E}\left[\|J_{i,k}\|\,\|J_{j,l}\|\sqrt{\sum_{m=1}^p\left(\sum_{r\in \mathbb{A}_{ij}}X_{r,m}\|J_{r,m}\|\mathbb{1}_{[1,\lambda_m]}(r)\right)^2}\right]\right\rbrace;\\
&\epsilon_4=\frac{1}{2}\sum_{k,l=1}^p\sum_{i=1}^{\lambda_k\wedge\lambda_l}\left|\frac{\Sigma_{k,l}}{\sqrt{\lambda_k\lambda_l}}-\mathbb{E}[X_{i,k}X_{i,l}]\right|;\\
&\epsilon_5=\frac{1}{2}\sum_{k,l=1}^p\sum_{i=1}^{\lambda_k}\sum_{j\in \mathbb{A} _i\setminus\lbrace i\rbrace}\left|\mathbb{E}[X_{i,k}X_{j,l}]\right|;\\
&\epsilon_6=\frac{6\sqrt{5}}{\sqrt{2\log 2}}\left(\sum_{i=1}^p \frac{\log\left(2\lambda_i\right)}{\lambda_i}\right)^{1/2}\left(\sum_{i=1}^p\Sigma_{i,i}\right)^{1/2};\\
&\epsilon_7=\sum_{k=1}^p\sum_{i=1}^{\lambda_k}\sqrt{\mathbb{E}\left[\left(X_{i,k}\right)^2\right]}\mathbb{E}\left\|J_{i,k}-\mathbb{1}_{[i/\lambda_k,1]}\right\|.
\end{align*}
\end{theorem}
\begin{remark}[Relevance of terms in the bound]~
\begin{enumerate}[label={(\alph*)}]
\item Terms $\epsilon_1, \epsilon_2, \epsilon_3$ correspond to a Berry-Esseen-type bound involving third moments of the summands, and also account for local dependence between the summands;
\item Terms $\epsilon_4$ and $\epsilon_5$ involve a variance estimation with the latter corresponding to the off-diagonal terms of the covariance matrix of the summands, accounting for the dependence;
\item Term $\epsilon_6$ comes from estimates on the moments of the Brownian modulus of continuity and accounts for the transition from the Skorokhod space to the Wiener space of continuous functions;
\item Term $\epsilon_7$ describes the randomness of the functions $J_{i,k}$ and their distance from indicators $\mathbb{1}_{[i/\lambda_k,1]}$. 
\end{enumerate}
\end{remark}
\begin{remark}[Convergence of the bound and process weak convergence]\label{remark_weak}
By Proposition \ref{prop_m}, if, in Theorem \ref{local3}, $J_{i,k}=\mathbb{1}_{[i/\lambda_k,1]}$ for all $i=1,\dots,n$ and $k=1,\dots,p$ and the bound $\sum_{i=1}^7\epsilon_i$ converges to $0$ faster than $\frac{1}{\log^2(\max(\lambda_1,\dots,\lambda_p))}$, then $\mathbf{Y}_n$ converges to $\mathbf{Z}$ in distribution with respect to the uniform topology. We note that, in practice, one might expect that $\epsilon_4$ and $\epsilon_5$ will be the slowest vanishing terms. 
\end{remark}
\begin{remark}[Independent summands]\label{remark_ind}
If the summands are independent in Theorem \ref{local3}, i.e. $\mathbb{A}_i=\lbrace i\rbrace$ for all $i$, then $\epsilon_2$ and $\epsilon_5$ disappear from the bound and $\epsilon_1$ and $\epsilon_3$ become simpler.
The new bound takes the following form
$$\left|\mathbb{E}g(\mathbf{Y}_n)-\mathbb{E}g(\mathbf{Z})\right|\leq \|g\|_{M^1}\left(\epsilon_1+\epsilon_3+\epsilon_4+\epsilon_6+\epsilon_7\right),$$
where:
\begin{align*}
&\epsilon_1=\frac{1}{6}\sum_{i=1}^n\mathbb{E}\left\lbrace\left[\sum_{k,l,m=1}^p\left(X_{i,k}X_{i,l}X_{i,m}\|J_{i,k}\|\,\|J_{i,l}\|\,\|J_{i,m}\|\mathbb{1}_{[1,\lambda_k]\cap[1,\lambda_l]\cap[1,\lambda_m]}(i)\right)^2\right]^{1/2}\right\rbrace;\\
&\epsilon_3=\frac{1}{3}\sum_{k,l=1}^p\sum_{i=1}^{\min(\lambda_k,\lambda_l)}\left\lbrace\vphantom{\left[\sqrt{\left(\sum_1^p\right)^2}\right]}\left|\mathbb{E}\left[X_{i,k}X_{i,l}\right]\right|\mathbb{E}\left[\|J_{i,k}\|\,\|J_{i,l}\|\sqrt{\sum_{m=1}^p\left(X_{i,m}\|J_{i,m}\|\mathbb{1}_{[1,\lambda_m]}(i)\right)^2}\right]\right\rbrace;\\
&\epsilon_4=\frac{1}{2}\sum_{k,l=1}^p\sum_{i=1}^{\min(\lambda_k,\lambda_l)}\left|\frac{\Sigma_{k,l}}{\sqrt{\lambda_k\lambda_l}}-\mathbb{E}[X_{i,k}X_{i,l}]\right|;\\
&\epsilon_6=\frac{6\sqrt{5}}{\sqrt{2\log 2}}\left(\sum_{i=1}^p \frac{\log\left(2\lambda_i\right)}{\lambda_i}\right)^{1/2}\left(\sum_{i=1}^p\Sigma_{i,i}\right)^{1/2};\\
&\epsilon_7=\sum_{k=1}^p\sum_{i=1}^{\lambda_k}\sqrt{\mathbb{E}\left[\left(X_{i,k}\right)^2\right]}\mathbb{E}\left\|J_{i,k}-\mathbb{1}_{[i/\lambda_k,1]}\right\|.
\end{align*}
In this case, it is also possible to derive a bound for the larger class of test functions $M$ (see Section \ref{section 2}). A bound for such test functions, in the case of independent summands, is obtained in Proposition \ref{local1}.
\end{remark}

\subsection{Scaled sum of independent vectors with dependent components}
The next result treats quantitatively the case of independent $p$-dimensional terms with dependent components, whose scaled sum can be compared to a correlated $p$-dimensional Brownian motion:
\begin{proposition}[Independent summands with dependent components]\label{local1}
Suppose that $X_1,...,X_n$, where $X_i=\left(X_i^{(1)},\dots,X_i^{(p)}\right)$ for $i=1,\dots,n$, are i.i.d. random vectors in $\mathbb{R}^p$. Suppose that each has a positive definite symmetric covariance matrix $\Sigma\in \mathbb{R}^{p\times p}$ and mean zero. Let:
$$\mathbf{Y}_n(t)=n^{-1/2}\sum_{i=1}^{\lfloor nt\rfloor} X_i,\quad t\in[0,1]$$
and for $\mathbf{B}$, a standard $p$-dimensional Brownian motion, let $\mathbf{Z}=\Sigma^{1/2}\mathbf{B}$. Then, for any $g\in M$:
\begin{align*}
&|\mathbb{E}g(\mathbf{Y}_n)-\mathbb{E}g(\mathbf{Z})|\\
\leq& \|g\|_Mn^{-1/2}\left\lbrace\sqrt{\log 2n}\left[\frac{6\sqrt{5}}{\sqrt{\pi\log 2}}\left(\sum_{i=1}^p\Sigma_{i,i}\right)^{1/2}+\frac{93p^{1/2}}{\sqrt{2\log 2}}\sum_{i=1}^p|\Sigma_{i,i}|^{3/2}\right]\right.\\
& +\frac{1}{6}\left(p^{1/2}\sum_{m=1}^p\mathbb{E}\left|X_1^{(m)}\right|^3+2\sum_{k,l=1}^p\left|\Sigma_{k,l}\right|\left(\sum_{m=1}^p\mathbb{E}\left|X_1^{(m)}\right|^2\right)^{1/2}\right)\\
&+\left. n^{-1}(\log 2n)^{3/2}p^{1/2}\frac{2160}{\sqrt{\pi}(\log 2)^{3/2}}\sum_{i=1}^p\left|\Sigma_{i,i}\right|^{3/2}\right\rbrace.
\end{align*}
\end{proposition}
\begin{remark}
The bound in Proposition \ref{local1} is of order $\frac{\sqrt{\log n}}{\sqrt{n}}$. We are not aware of any reference providing a bound in a similar setup (i.e. in a multidimensional version of Donsker's theorem) but we note that our bound is of the same order as the bound derived in \cite{diffusion} for one-dimensional Donsker's theorem.
\end{remark}
\begin{remark}
If the components are uncorrelated and scaled in Proposition \ref{local1}, i.e. $\Sigma=I_{p\times p}$, then the bound simplifies in the following way:
\begin{align*}
&|\mathbb{E}g(\mathbf{Y}_n)-\mathbb{E}g(\mathbf{Z})|\\
\leq& \|g\|_Mn^{-1/2}\left\lbrace\sqrt{\log 2n}\left[\frac{6\sqrt{5}p^{1/2}}{\sqrt{2\log 2}}+\frac{93p^{3/2}}{\sqrt{\pi\log 2}}\right]\right.\\
&+\frac{1}{6}\left(p^{1/2}\sum_{m=1}^p\mathbb{E}\left|X_1^{(m)}\right|^3+2p^{3/2}\right)
+\left. n^{-1}(\log 2n)^{3/2}p^{3/2}\frac{2160}{\sqrt{\pi}(\log 2)^{3/2}}\right\rbrace.
\end{align*}
\end{remark}
\begin{remark}
For fixed $p$, by Proposition \ref{prop_m}, Theorem \ref{local1} implies that $\mathbf{Y}_n$ converges in distribution to $ \mathbf{Z}$ in the uniform topology as the bound is of order $\frac{\sqrt{\log n}}{\sqrt{n}}$. If one made $p$ depend on $n$ the bound would also converge to zero as $n\to\infty$ as long as $p=o\left(n^{1/5}\right)$.
\end{remark}

\subsection{Non-degenerate bivariate U-statistics}
The next result will be proved using ideas similar to those used to prove Theorem \ref{local3}.
It treats non-degenerate bivariate U-statistics. Those, as observed for instance in \cite[Corollary 1]{hall}, after proper rescaling, represent a process created out of globally dependent summands and converge to standard Brownian motion in distribution under certain conditions. We find a bound for the rate of this convergence.
 
We note that bivariate U-statistics are defined to be random variables of the form:
$$S_n^2(h)=\sum_{1\leq i_1<i_2\leq n} h(X_{i_1},X_{i_2}),\quad n\geq 1$$
for a symmetric real (or complex) function $h$ on $\mathcal{S}^2$ (where $\mathcal{S}$ is some measurable space) and a sequence of i.i.d. random variables $(X_i)_{i\geq 1}$ taking values in $\mathcal{S}$. Here, we only consider non-degenerate U-statistics, i.e. those with $0<\sigma_w^2=\text{Var}(w(X_1))<\infty$, where $w(x)=\mathbb{E}[h(X_1,x)]$. The reason is that in the case of degenerate ones (i.e. those satisfying $\text{Var}(w(X_1))=0$) the limit in the invariance principle is non-Gaussian (see \cite[Corollary 1]{hall}), which is beyond the scope of this paper.
\begin{theorem}[Non-degenerate bivariate U-statistics]\label{theorem_u_stats}
Let $X_1,X_2,...$ be i.i.d. random variables taking values in some measurable space $\mathcal{S}$ and let $h:\mathcal{S}^2\to\mathbb{R}$ be a symmetric function such that $\mathbb{E}\left[h(X_1,X_2)\right]=0$, $\mathbb{E}\left[h^2(X_1,X_2)\right]=\sigma_h^2<\infty$. Also, suppose that, for the function $w(x)=\mathbb{E}[h(X_1,x)]$, we have that: $0<\sigma_w^2=\text{Var}(w(X_1))$ and $\mathbb{E}|w(X_1)|^3<\infty$. Let:
$$\mathbf{Y}_n(t)=\frac{n^{-3/2}}{\sigma_wt}\sum_{1\leq i_1<i_2\leq \lfloor nt\rfloor}h(X_{i_1},X_{i_2}),\quad t\in[0,1]$$
and let $\mathbf{Z}$ be a standard Brownian motion. Then, for any $g\in M^2$, as defined by (\ref{m_2}):
\begin{align*}
|\mathbb{E}g(\mathbf{Y}_n)-\mathbb{E}g(\mathbf{Z})|
\leq& \|g\|_{M^2}n^{-1/2}\left[ \left(141+16\frac{\sigma_h^2}{\sigma_w^2}+12\left(\frac{\sigma_h^2}{\sigma_w^2}-2\right)^{1/2}\right)\sqrt{\log 3n}
\right.\\
&\left.+43+\frac{\mathbb{E}|w(X_1)|^3+2\sigma_w^2\mathbb{E}|w(X_1)|}{6\sigma_w^3}\vphantom{\left(\frac{\sigma_h^2}{\sigma_w^2}-2\right)^{1/2}}\right].
\end{align*}
\end{theorem}
\begin{remark}[Discussion of the bound]
The term $\frac{\mathbb{E}|w(X_1)|^3+2\sigma_w^2\mathbb{E}|w(X_1)|}{6\sigma_w^3}$ appearing in the bound comes from the comparison of the process given by
$$\tilde{\mathbf{Y}}_n(t)=\frac{n^{-3/2}}{\sigma_wt}\sum_{1\leq i_1<i_2\leq \lfloor nt\rfloor}\left(w(X_{i_1})+w(X_{i_2})\right),\quad t\in[0,1]$$ and a piecewise constant Gaussian process. It involves a Berry-Esseen-type third absolute moment component. The remaining terms come from the comparison of $\mathbf{Y}_n$ and $\tilde{\mathbf{Y}}_n$ and from the comparison of the piecewise constant Gaussian process and Brownian motion, for
 which the Brownian modulus of continuity is used.
 
 The bound is of order $\frac{\sqrt{\log n}}{\sqrt{n}}$. We are not aware of any reference providing a bound on the rate of \textit{functional} convergence of non-degenerate U-statistics but we note that our bound is of the same order as the bound obtained in \cite{diffusion} for the rate of convergence in the classical Donsker's theorem.
\end{remark}
\begin{remark}
By Proposition \ref{prop_m}, Theorem \ref{theorem_u_stats} implies that $\mathbf{Y}_n$ converges in distribution to  $\mathbf{Z}$ in the uniform (and Skorokhod) topology.
\end{remark}

\begin{remark}
The constants in Theorems \ref{local3}, \ref{theorem_u_stats} and Proposition \ref{local1} are not optimal ones as they are often estimated in a crude manner in the proofs presented in the section below. The constants are, however, expressed explicitly, which is often not the case in related pieces of literature. We also have no information about the optimality of the orders of the obtained bounds.  
\end{remark}
\section{Example: Exceedances of the m-scans process}\label{ex_exceedances}
Consider an extension of the one-dimensional results presented in \cite[Example 9.2, p. 254]{normal_approx} to the multidimensional and functional setting. For $j=1,2,\dots,$ let $V_j=\left(V_{j,1},\dots,V_{j,p}\right)$  be i.i.d. random vectors in $\mathbb{R}^p$. For $k=1,\dots, p$ and $i=1,2,\dots$ let $R_{i,k}=\sum_{l=0}^{m-1}V_{i+l,k}$ be an $m$-scans process. Let $a=(a_1,\dots,a_p)\in\mathbb{R}^p$ and suppose that $n>m$.

For $k=1,\dots,p$, let $\pi_k=\mathbb{P}(R_{1,k}\leq a_k)$ and for $i=1,\dots,n$ and $k=1,\dots,p$, let
$$X_{i,k}=\frac{1}{n}\left(\sum_{j=1}^n\mathbb{1}[R_{n(i-1)+j,k}\leq a_k]\right)-\pi_k.$$ 

Extending \cite[(4.1)]{dembo_rinott}, we have that, for $k,l=1,\dots,p$ and for $\psi_{k,l}(d)=\mathbb{P}\left[R_{d+1,k}\leq a_k,R_{1,l}\leq a_l\right]-\pi_k\pi_l$,
\begin{equation}\label{same_vector}
\mathbb{E}\left[X_{i,k}X_{i,l}\right]=\frac{1}{n}\left(\psi_{k,l}(0)+\sum_{d=1}^{m-1}\left(1-\frac{d}{n}\right)\left(\psi_{l,k}(d)+\psi_{k,l}(d)\right)\right).
\end{equation}

Let $X_i=(X_{i,1},\dots,X_{i,p})$ for $i=1,\dots,n$. Note that  $\mathbb{A}_{i}=\lbrace i-1,i,i+1\rbrace$ satisfies the requirement that $X_i$ is independent of $\lbrace X_j:j\in\mathbb{A}_i^c\rbrace$ and that we can take $\mathbb{A}_{ij}=\mathbb{A}_i\cup\mathbb{A}_j$. Furthermore,  for all $k,l\in\lbrace 1,\dots, p\rbrace$,
\begin{equation}\label{diff_vector}
\mathbb{E}\left[X_{i,k}X_{i+1,l}\right]=\frac{1}{n^2}\sum_{d=1}^{m-1}d\psi_{k,l}(d).
\end{equation}
Consider
$$\mathbf{Y}_n(t)=\sum_{i=1}^{\lfloor nt\rfloor}\left(X_{i,1},\dots,X_{i,p}\right)\quad t\in[0,1].$$
Let $\Sigma\in\mathbb{R}^{p\times p}$ be given by
\begin{equation}\label{cov_matrix}
\Sigma_{k,l}=\psi_{k,l}(0)+\sum_{d=1}^{m-1}\left(\psi_{l,k}(d)+\psi_{k,l}(d)\right).
\end{equation}
We will bound the distance between $\mathbf{Y}_n$ and $\mathbf{Z}=\Sigma^{1/2}\mathbf{B}$, where $\mathbf{B}$ is a standard $p$-dimensional Brownian motion. Using the notation of Theorem \ref{local3}, note that for all $k\in\lbrace 1,\dots,p\rbrace$, $\lambda_k=n$, for all $i\in\lbrace 1,\dots, n\rbrace$, $J_{i,k}=\mathbb{1}_{[i/n,1]}$ and
\begin{enumerate}[label=(\arabic*)]
\item By Cauchy-Schwarz and Jensen inequalities and (\ref{same_vector}),
\begin{flalign*}
\epsilon_1\leq&\frac{3}{2n^{1/2}}\sum_{k,l,r=1}^p\left\lbrace \left(\psi_{k,k}(0)+2\sum_{d=1}^{m-1}\left(1-\frac{d}{n}\right)\left(\psi_{k,k}(d)\right)\right)^{1/2}\right.&\\
&\cdot\left(\psi_{l,l}(0)+2\sum_{d=1}^{m-1}\left(1-\frac{d}{n}\right)\left(\psi_{l,l}(d)\right)\right)^{1/2}&\\
&\cdot\left.\left(\psi_{r,r}(0)+2\sum_{d=1}^{m-1}\left(1-\frac{d}{n}\right)\left(\psi_{r,r}(d)\right)\right)^{1/2}\right\rbrace;&
\end{flalign*}
\item By Cauchy-Schwarz and Jensen inequalities and (\ref{same_vector}),
\begin{flalign*}
\epsilon_2\leq &\frac{2}{3n^{1/2}}\sum_{k,l,r=1}^p\left\lbrace \left(\psi_{k,k}(0)+2\sum_{d=1}^{m-1}\left(1-\frac{d}{n}\right)\left(\psi_{k,k}(d)\right)\right)^{1/2}\right.&\\
&\cdot\left(\psi_{l,l}(0)+2\sum_{d=1}^{m-1}\left(1-\frac{d}{n}\right)\left(\psi_{l,l}(d)\right)\right)^{1/2}&\\
&\cdot\left.\left(\psi_{r,r}(0)+2\sum_{d=1}^{m-1}\left(1-\frac{d}{n}\right)\left(\psi_{r,r}(d)\right)\right)^{1/2}\right\rbrace;&
\end{flalign*}
\item By Cauchy-Schwarz and Jensen inequalities and (\ref{same_vector}),
\begin{flalign*}
\epsilon_3\leq&\frac{2}{n^{1/2}}\sum_{k,l,r=1}^p\left\lbrace \left(\psi_{k,k}(0)+2\sum_{d=1}^{m-1}\left(1-\frac{d}{n}\right)\left(\psi_{k,k}(d)\right)\right)^{1/2}\right.&\\
&\cdot\left(\psi_{l,l}(0)+2\sum_{d=1}^{m-1}\left(1-\frac{d}{n}\right)\left(\psi_{l,l}(d)\right)\right)^{1/2}&\\
&\cdot\left.\left(\psi_{r,r}(0)+2\sum_{d=1}^{m-1}\left(1-\frac{d}{n}\right)\left(\psi_{r,r}(d)\right)\right)^{1/2}\right\rbrace;&
\end{flalign*}
\item By (\ref{same_vector}) and (\ref{cov_matrix}),
\begin{flalign*}
&\epsilon_4=\frac{1}{2n}\sum_{k,l=1}^p\left|\sum_{d=1}^{m-1}d(\psi_{l,k}(d)+\psi_{k,l}(d))\right|;&
\end{flalign*}
\item By (\ref{diff_vector}),
\begin{flalign*}
&\epsilon_5\leq\frac{1}{n}\sum_{l,k=1}^p\sum_{d=1}^{m-1}d\psi_{k,l}(d);&
\end{flalign*}
\item By (\ref{cov_matrix}),
\begin{flalign*}
&\epsilon_6=\frac{6\sqrt{5}p^{1/2}}{\sqrt{2\log 2}}\frac{\sqrt{\log(2n)}}{\sqrt{n}}\left[\sum_{k=1}^p\left(\psi_{k,k}(0)+2\sum_{d=1}^{m-1}\psi_{k,k}(d)\right)\right]^{1/2};&
\end{flalign*}
\item Since for all $k\in\lbrace 1,\dots,p\rbrace$ and $i\in\lbrace 1,\dots, n\rbrace$, $J_{i,k}=\mathbb{1}_{[i/n,1]}$,
\begin{flalign*}
&\epsilon_7=0.&
\end{flalign*}
\end{enumerate}
By Theorem \ref{local3}, for any $g\in M^1$, as defined in (\ref{m_1}),
$$|\mathbb{E}g(\mathbf{Y}_n)-\mathbb{E}g(\mathbf{Z})|\leq\|g\|_{M^1}\sum_{i=1}^7\epsilon_i,$$
which gives the desired bound. The bound clearly approaches zero faster than $\log^{-2}(n)$, as $n\to\infty$. Indeed, terms $\epsilon_1,\epsilon_2,\epsilon_3$ converge to zero at rate $n^{-1/2}$, $\epsilon_4$ and $\epsilon_5$ do so at rate $n^{-1}$, $\epsilon_6$ at rate $\frac{\sqrt{\log n}}{\sqrt{n}}$ and $\epsilon_7=0$. This, by Proposition \ref{prop_m}, implies that $\mathbf{Y}_n$ converges in distribution to $\mathbf{Z}$ with respect to the uniform topology.

\section{Proofs of the main results}\label{section5}
The main tool used in the proofs of Theorems \ref{local3}, \ref{theorem_u_stats} and Proposition \ref{local1} is Stein's method. It can be used in a surprisingly easy way to find a distance of the processes of interest from certain scaled sums of Gaussian random variables, which approximate the limiting continuous Gaussian process. 

First, we set up Stein's method for distributions of certain $D^p$-valued random objects expressed as scaled sums of Gaussian random variables. Using a collection of Ornstein-Uhlenbeck processes with a Gaussian stationary law, we will construct a process whose stationary law is that of our target distribution. Then, we will find the infinitesimal generator $\mathcal{A}$ of that process and deduce that $\mathcal{A}g=g-\mathbb{E}_\mu g$ can be used as our Stein equation, where $\mu$ is the target law. This follows from the fact that $\mathbb{E}_\mu\mathcal{A}g=0$ for all $g$ in the domain of $\mathcal{A}$. We will then solve the Stein equation for all $g\in M$, using the analysis of \cite{kasprzak}, and use some appealing properties of the Ornstein-Uhlenbeck semigroup to prove bounds on the derivatives of the solution.
\subsection{Setting up Stein's method}\label{section_stein}
Let $n,p\in\mathbb{N}_+$ and let $\tilde{Z}_{i,k}$'s be centred Gaussian random variables for $i=1,\dots,n$, $k=1,\dots,p$. Suppose that
\begin{enumerate}[label=\alph*)]
\item the covariance matrix of $\left(\tilde{Z}_{1,1},\dots, \tilde{Z}_{1,p},\tilde{Z}_{2,1},\dots,\tilde{Z}_{2,p},\dots, \tilde{Z}_{n,1},\dots, \tilde{Z}_{n,p}\right)$ is given by $\Sigma_n\in\mathbb{R}^{(np)\times(np)}$;
\item $\lbrace J_{i,k}\in D\left([0,1],\mathbb{R}\right):\,i=1,\dots, n,\,k=1,\dots, p\rbrace$ is a collection of functions independent of $\lbrace \tilde{Z}_{i,k}:\,i=1,\dots, n,\,k=1,\dots, p\rbrace$;
\item $\lambda_k\leq n$, for all $k=1,\dots,p$.
\end{enumerate}
Let
\begin{equation}\label{d_n}
\mathbf{D}_n(t)=\left(\sum_{i=1}^{ \lambda_1}\tilde{Z}_{i,1}J_{i,1}(t),\dots,\sum_{i=1}^{\lambda_p}\tilde{Z}_{i,p}J_{i,p}(t)\right),\quad t\in[0,1],
\end{equation}

Now let $\lbrace (\mathscr{X}_{i,j}(u),u\geq 0):i=1,\dots,n,j=1,\dots,p\rbrace$ be an array of i.i.d. Ornstein-Uhlenbeck processes with stationary law $\mathcal{N}(0,1)$, i.e. independent processes such that each weakly solves the following stochastic differential equation
$$dx_t=-x_t\,dt+\sqrt{2}\,dB(t),\quad x_0\sim\mathcal{N}(0,1),$$
for $B(t),t\geq 0$ denoting the standard Wiener process.
Suppose that the collection $\lbrace (\mathscr{X}_{i,j}(u),u\geq 0):i=1,\dots,n,j=1,\dots,p\rbrace$ is independent of the collection $\lbrace J_{i,k}: i=1,\dots,n,k=1,\dots,p\rbrace$. Consider:
$$\tilde{\mathscr{U}}(u)=\left(\Sigma_n\right)^{1/2}\left(\mathscr{X}_{1,1}(u),\dots,\mathscr{X}_{1,p}(u),\mathscr{X}_{2,1}(u),\dots,\mathscr{X}_{2,p}(u),\dots,\mathscr{X}_{n,1}(u),\dots \mathscr{X}_{n,p}(u)\right)^T$$ for $u\geq 0$ and write $\mathscr{U}_{i,k}(u)=\left(\tilde{\mathscr{U}}(u)\right)_{p(i-1)+k}$ for $i=1,\dots,n$ and $k=1,\dots,p$. This notation is introduced for convenience, in order to define the following process:
$$\mathbf{W}_n(t,u)=\left(\sum_{i=1}^{\lambda_1}\mathscr{U}_{i,1}(u)J_{i,1}(t),\dots,\sum_{i=1}^{ \lambda_p}\mathscr{U}_{i,p}(u)J_{i,p}
(t)\right),\qquad t\in[0,1],\quad u\geq 0.$$
The stationary law of the process $\left(\mathbf{W}_n(\cdot,u)\right)_{u\geq 0}$ is exactly the law of $\mathbf{D}_n$. We claim that:
\begin{proposition}\label{prop12.7}
The infinitesimal generator $\mathcal{A}_n$ of the process $\left(\mathbf{W}_n(\cdot,u)\right)_{u\geq 0}$ acts on any $f\in M$ in the following way:
\begin{align*}
&\mathcal{A}_nf(w)=-Df(w)[w]+\mathbb{E}D^2f(w)\left[\mathbf{D}_n,\mathbf{D}_n\right].
\end{align*}
\end{proposition}
\begin{remark}
By definition, the first Fr\'echet derivative of a function, at a certain point, is a linear map, while the second Fr\'echet derivative of a function, at a certain point, is a bilinear map. In Proposition \ref{prop12.7} above, and throughout this paper, $Df(w)[w]$ denotes the first Fr\'echet derivative of $f$, at $w$, applied to $w$ and $D^2f(w)\left[\mathbf{D}_n,\mathbf{D}_n\right]$ is the second Fr\'echet derivative of $f$, at $w$, applied to $\mathbf{D}_n$ and $\mathbf{D}_n$.
\end{remark}
\begin{remark}
The generator in Proposition \ref{prop12.7} can also be written in the following way:
\begin{align*}
\mathcal{A}_nf(w)=-Df(w)[w]+\sum_{k,l=1}^p\sum_{i=1}^{\lambda_k}\sum_{j=1}^{\lambda_l}\left(\Sigma_n\right)_{p(i-1)+k,p(j-1)+l}\mathbb{E}D^2f(w)\left[e_k J_{i,k},e_lJ_{j,l}\right].
\end{align*}
\end{remark}
Let us prove a lemma that will be used in the proof of Proposition \ref{prop12.7}.
\begin{lemma}\label{lemma_1.2}
We have, for $u\geq 0,v\geq 0$:
$$\mathbf{W}_n(\cdot,u+v)-e^{-v}\mathbf{W}_n(\cdot,u)\stackrel{\mathcal{D}}=\sigma(v)\mathbf{D}_n(\cdot)$$
for $\sigma^2(v)=1-e^{-2v}$.
\end{lemma}
\begin{proof}
We can construct i.i.d. standard Brownian motions $\mathscr{B}_{i,j}$ such that $\left(\mathscr{X}_{i,j}(u),u\geq 0\right)=\left(e^{-u}\mathscr{B}_{i,j}(e^{2u}),u\geq 0\right)$ (see, for instance \cite[Subsection 4.4.3]{pitman_yor}). Then, writing $\mathbf{W}_n=\left(\mathbf{W}_n^{(1)},\dots,\mathbf{W}_n^{(p)}\right)$ and $\mathbf{D}_n=\left(\mathbf{D}_n^{(1)},\dots,\mathbf{D}_n^{(k)}\right)$ we obtain for all $k=1,\dots,p$:
\begin{align*}
&\mathbf{W}_n^{(k)}(\cdot,u+v)-e^{-v}\mathbf{W}_n^{(k)}(\cdot,u)\\
=&\sum_{i=1}^{\lambda_k}\left[\mathscr{U}_{i,k}(u+v)-e^{-v}\mathscr{U}_{i,k}(u)\right]J_{i,k}(\cdot)\\
=&\sum_{i=1}^{\lambda_k}\left[\left(\tilde{\mathscr{U}}(u+v)\right)_{p(i-1)+k}-e^{-v}\left(\tilde{\mathscr{U}}(u)\right)_{p(i-1)+k}\right]J_{i,k}(\cdot)\\
\stackrel{(\ast)}=&\sum_{j=1}^n\sum_{l=1}^p\sum_{i=1}^{\lambda_k}\left(\Sigma_n^{1/2}\right)_{p(i-1)+k,p(j-1)+l}\left[\mathscr{X}_{j,l}(u+v)-e^{-v}\mathscr{X}_{j,l}(u)\right]J_{i,k}(\cdot)\\
\stackrel{\mathcal{D}}=&e^{-(u+v)}\sum_{j=1}^n\sum_{l=1}^p\sum_{i=1}^{\lambda_k}\left(\Sigma_n^{1/2}\right)_{p(i-1)+k,p(j-1)+l}\left[\mathscr{B}_{j,l}\left(e^{2(u+v)}\right)-\mathscr{B}_{j,l}\left(e^{2u}\right)\right]J_{i,k}(\cdot)\\
\stackrel{\mathcal{D}}=&\sigma(v)\mathbf{D}^{(k)}_n(\cdot),
\end{align*}
as $\mathscr{B}_{j,l}\left(e^{2(u+v)}\right)-\mathscr{B}_{j,l}\left(e^{2u}\right)\sim\mathcal{N}\left(0,e^{2(u+v)}-e^{2u}\right)$. In the above formula, the equality $(\ast)$ represents the matrix multiplication formula.
\end{proof}
\begin{proof}[Proof of Proposition \ref{prop12.7}]
Note that the semigroup of $\left(\mathbf{W}_n(\cdot,u)\right)_{u\geq 0}$, acting on $L$ of Section \ref{section 2} is defined by:
\begin{equation}\label{semigroup}
(T_{n,u}f)(w):=\mathbb{E}\left[\left.f\left(\mathbf{W}_n(\cdot,u)\right)\right|\mathbf{W}_n(\cdot,0)=w\right]=\mathbb{E}\left[f\left(we^{-u}+\sigma(u)\mathbf{D}_n(\cdot)\right)\right],
\end{equation}
where the last equality follows from Lemma \ref{lemma_1.2}.
By (\ref{semigroup}) and Lemma \ref{first_der} we have that, for every $f\in M$:
\begin{align*}
&\left|\vphantom{\frac{1}{2}}(T_{n,u}f)(w)-f(w)-\mathbb{E}Df(w)[\sigma(u)\mathbf{D}_n-w(1-e^{-u})]\right.\\
&\left.-\frac{1}{2}\mathbb{E}D^2f(w)\left[\sigma(u)\mathbf{D}_n-w(1-e^{-u}),\sigma(u)\mathbf{D}_n-w(1-e^{-u})\right]\right|\\
\leq&\|f\|_M\mathbb{E}\|\sigma(u)\mathbf{D}_n-w(1-e^{-u})\|^3\\
\leq &K_1(1+\|w\|^3)u^{3/2}
\end{align*}
for a constant $K_1$ depending only on $f$, where the last inequality follows from the fact that for $u\geq 0$, $\sigma^3(u)\leq 3u^{3/2}$ and $(1-e^{-u})^3\leq u^{3/2}$.
So:
\begin{align}
&\left|(T_{n,u}f-f)(w)+uDf(w)[w]-u\mathbb{E}D^2f(w)[\mathbf{D}_n,\mathbf{D}_n]\right| \nonumber\\
\leq&\left|\vphantom{\frac{1}{2}}(T_{n,u}f)(w)-f(w)-\mathbb{E}Df(w)[\sigma(u)\mathbf{D}_n-w(1-e^{-u})]\right.\nonumber\\
&-\left.\frac{1}{2}\mathbb{E}D^2f(w)[\sigma(u)\mathbf{D}_n-w(1-e^{-u}),\sigma(u)\mathbf{D}_n-w(1-e^{-u})]\right|+\left|\sigma(u)\mathbb{E}Df(w)[\mathbf{D}_n]\right|\nonumber\\
&+\left|(u-1+e^{-u})Df(w)[w]\right|+\left|\left(\frac{\sigma^2(u)}{2}-u\right)\mathbb{E}D^2f(w)[\mathbf{D}_n,\mathbf{D}_n]\right|\nonumber\\
&+\left|\frac{(1-e^{-u})^2}{2}D^2f(w)[w,w]\right|+\left|\sigma(u)(1-e^{-u})\mathbb{E}D^2f(w)[\mathbf{D}_n,w]\right|\nonumber\\
\leq &K_2u^{3/2}\left[\vphantom{\sum}(1+\|w\|^3)+(1+\|w\|^2)\|w\|+(1+\|w\|)\mathbb{E}\|\mathbf{D}_n\|^2\right.\nonumber\\
&\left.+(1+\|w\|)\|w\|^2+(1+\|w\|)\|w\|\mathbb{E}\|\mathbf{D}_n\|\vphantom{\sum}\right]+\left|\sigma(u)\mathbb{E}Df(w)[\mathbf{D}_n]\right|\nonumber\\
\leq& K_3(1+\|w\|^3)u^{3/2}\text{,}\label{trala}
\end{align}
for some constants $K_2$ and $K_3$ depending only on $f$. The last inequality follows from the fact that:
$$\mathbb{E}Df(w)[\mathbf{D}_n]=\sum_{k=1}^p\sum_{i=1}^{\lambda_k}\mathbb{E}Df(w)\left[J_{i,k}e_k\right]\mathbb{E}[\tilde{Z}_{i,k}]=0\text{.}$$
Therefore, by (\ref{trala}), we obtain that:
$$\mathcal{A}_nf(w):=\lim_{u\searrow 0}\frac{T_{n,u}f(w)-f(w)}{u}=-Df(w)[w]+\mathbb{E}D^2f(w)\left[\mathbf{D}_n,\mathbf{D}_n\right]\text{,}$$
as required.
\end{proof}
Now we prove the following:
\begin{proposition}\label{prop12.17}
For any $g\in M$ such that $\mathbb{E}g(\mathbf{D}_n)=0$, the Stein equation $\mathcal{A}_nf_n=g$ is solved by:
\begin{equation}\label{phi}
f_n=\phi_n(g)=-\int_0^{\infty}T_{n,u}gdu,
\end{equation}
where $(T_{n,u}f)(w)=\mathbb{E}\left[f\left(we^{-u}+\sigma(u)\mathbf{D}_n\right)\right]$
for $\sigma^2(v)=1-e^{-2v}$. Furthermore:
\begin{align}
\text{A)} \quad &\|D\phi_n(g)(w)\|\leq \|g\|_{ M}\left(1+\frac{2}{3}\|w\|^2+\frac{4}{3}\mathbb{E}\|\mathbf{D}_n\|^2\right)\text{,}\nonumber\\
\text{B)} \quad &\|D^2\phi_n(g)(w)\|\leq \|g\|_{ M}\left(\frac{1}{2}+\frac{\|w\|}{3}+\frac{\mathbb{E}\|\mathbf{D}_n\|}{3}\right)\text{,}\nonumber\\
\text{C)} \quad &\frac{\left\|D^2\phi_n(g)(w+h)-D^2\phi_n(g)(w)\right\|}{\|h\|}\nonumber\\
\leq&\sup_{w,h\in D^p}\frac{\|D^2(g+c)(w+h)-D^2(g+c)(w)\|}{3\|h\|}.\label{norm_bound}
\end{align}
for any constant function $c:D^p\to\mathbb{R}$ and for all $w,h\in D^p$.
\end{proposition}
\begin{remark}
It is worth noting that obtaining a bound for $\mathbb{E}\|\mathbf{D}_n\|$ or $\mathbb{E}\|\mathbf{D}_n\|^2$ that does not blow up with $n\to\infty$ is not easy, unless $\mathbf{D}_n$ is a martingale and Doob's $L^2$ inequality can be used to show that $\mathbb{E}\|\mathbf{D}_n\|^2\leq\mathbb{E}|\mathbf{D}_n(1)|=\mathbb{E}\sqrt{\sum_{i=1}^p \mathbf{D}_n^{(i)}(1)}$. This is, for instance, the case, if $\tilde{Z}_i=\left(\tilde{Z}_{i,1},\dots,\tilde{Z}_{i,p}\right)$'s are independent and $J_{i,k}$'s are independent.
\end{remark}
\begin{proof}
The first part of the proposition follows by the argument used to prove \cite[Proposition 4.4]{kasprzak} upon noting that we can readily substitute $\mathbf{D}_n$ in the place of $Z$ therein due to $\mathbb{E}\|\mathbf{D}_n\|^3$ being finite. What follows is a sketch summary of this argument. Using dominated convergence theorem, we note that, for any $f\in M$ and $w\in D([0,1],\mathbb{R})$,
\begin{align*}
\left(\frac{d}{ds}\right)^+T_{n,s}f(w)&=\lim_{h\searrow 0}T_{n,s}\left[\frac{T_{n,h}-I}{h}f(w)\right]=\lim_{h\searrow 0}\mathbb{E}\left[\frac{T_{n,h}-I}{h}f(we^{-s}+\sigma(s)\mathbf{D}_n)\right]\\
&=\mathbb{E}\left[\lim_{h\searrow 0}\frac{T_{n,h}-I}{h}f(we^{-s}+\sigma(s)\mathbf{D}_n)\right]=T_{n,s}\mathcal{A}_nf(w).
\end{align*}
Similarly, for $s>0$, $\left(\frac{d}{ds}\right)^-T_{n,s}f=T_{n,s}\mathcal{A}_nf$ because:
\begin{align*}
&\lim_{h\searrow0}\frac{1}{-h}\left[T_{n,s-h}f-T_{n,s}f\right](w)-T_{n,s}\mathcal{A}_nf(w)\\
=&\lim_{h\searrow0}T_{n,s-h}\left[\left(\frac{T_{n,h}-I}{h}-\mathcal{A}_n\right)f\right](w)+\lim_{h\searrow0}\left(T_{n,s-h}-T_{n,s}\right)\mathcal{A}_nf(w)\\
=&\lim_{h\searrow 0}\mathbb{E}\left[\left(\frac{T_{n,h}-I}{h}-\mathcal{A}_n\right)f(we^{-s+h}+\sigma(s-h)\mathbf{D}_n)\right]\\
&+\lim_{h\searrow 0}\mathbb{E}\left[\mathcal{A}_n f(we^{-s+h}+\sigma(s-h)\mathbf{D}_n)-\mathcal{A}_nf(we^{-s}+\sigma(s)\mathbf{D}_n)\right]\\
=&0
\end{align*}
again, by dominated convergence and an argument similar to (\ref{trala}). Thus, for all $f\in M$ and $s>0$, we have
$$\frac{d}{ds}T_{n,s}f=T_{n,s}\mathcal{A}_nf$$
and so, by the fundamental theorem of calculus, for any $r>0$,
$$T_{n,r}f-f=\int_0^rT_{n,s}\mathcal{A}_nf\,ds.$$
Applying this to $f=\int_0^t T_{n,u}g\, du$ (which belongs to $M$, for instance by \cite[(2.23), (2.24)]{diffusion}), for some $t>0$, we obtain for any $r>0$ and any $w\in D([0,1],\mathbb{R})$,
\begin{align}\label{add_proof2}
T_{n,r}\int_0^t T_{n,u}g(w)du-\int_0^t T_{n,u}g(w)du=\int_0^r T_{n,s}\mathcal{A}_n\left(\int_0^tT_{n,u}g(w)du\right)ds.
\end{align}

On the other hand, for all $w\in D[0,1]$ and $h>0$:
\begin{align}
&\frac{1}{h} [T_{n,h}-I]\int_0^tT_{n,u}g(w)du=\frac{1}{h}\int_0^t[T_{n,u+h}g(w)-T_{n,u}g(w)]du\nonumber\\
=&\frac{1}{h}\int_t^{t+h}T_{n,u}g(w)du-\frac{1}{h}\int_0^hT_{n,u}g(w)du\nonumber\\
\stackrel{(\ref{semigroup})}=&\frac{1}{h}\int_t^{t+h}\mathbb{E}[g(w e^{-u}+\sigma(u)\mathbf{D}_n)]du-\frac{1}{h}\int_0^{h}\mathbb{E}[g(w e^{-u}+\sigma(u)\mathbf{D}_n)]du.
&\label{add_proof1}
\end{align}
Taking $h\to 0$ in (\ref{add_proof1}) and noting that
$$\lim_{h\searrow 0}\left[\frac{1}{h}\int_0^h\mathbb{E}g\left(we^{-s}+\sigma(s)\mathbf{D}_n\right)ds\right]=g(w),$$
as proved in \cite[(4.6)]{kasprzak}, yields
\begin{align}\label{add_proof3}
\mathcal{A}_n\left(\int_0^tT_{n,u}gdu\right)=T_{n,t}g-g.
\end{align}

Now, taking $t\to\infty$ in (\ref{add_proof2}) and applying dominated convergence, we obtain
\begin{align}
T_{n,r}\int_0^{\infty}T_{n,u}g(w)-\int_0^{\infty} T_{n,u}g(w)du=&\int_0^r T_{n,s}\lim_{t\to\infty}\mathcal{A}_n\left(\int_0^t T_{n,u}g(w)du\right)ds\nonumber\\
\stackrel{(\ref{add_proof3})}=&-\int_0^rT_{n,s}g(w)ds.\label{add_proof}
\end{align}

Furthermore, by \cite[Lemma 4.1]{kasprzak}, $\int_0^{\infty}T_{n,u}gdu$ is in the domain of $\mathcal{A}_n$. Therefore,  dividing both sides of (\ref{add_proof}) by $r$ and taking $r\searrow 0$ gives
\begin{align*}
\mathcal{A}_n\left(\int_0^{\infty}T_{n,u}g(w)du\right)\stackrel{(\ref{add_proof})}=&-\lim_{r\searrow 0}\frac{1}{r}\int_0^r T_{n,s}g(w)ds\\
=&-\lim_{r\searrow 0}\left[\frac{1}{r}\int_0^r\mathbb{E}g\left(we^{-s}+\sigma(s)\mathbf{D}_n\right)ds\right]\\
&=-g(w),
\end{align*}
where the last equality follows from \cite[(4.6)]{kasprzak}. This lets us conclude that the Stein equation $\mathcal{A}_nf_n=g$ is indeed solved by:
\begin{equation*}
f_n=\phi_n(g)=-\int_0^{\infty}T_{n,u}gdu.
\end{equation*}
Now, note that for $\phi_n$ defined in (\ref{phi}) we get:
\begin{align*}
&\phi_n(g)(w+h)-\phi_n(g)(w)\\
\stackrel{(\ref{semigroup})}=&-\mathbb{E}\int_0^\infty\left[g\left((w+h)e^{-u}+\sigma(u)\mathbf{D}_n\right)-g\left(we^{-u}+\sigma(u)\mathbf{D}_n\right)\right]du
\end{align*}
and so dominated convergence (which can be applied because of \cite[(4.2)]{kasprzak}) gives:
\begin{equation}\label{4.800}
D^k\phi_n(g)(w)=-\mathbb{E}\int_0^{\infty}e^{-ku}D^kg(we^{-u}+\sigma(u)\mathbf{D}_n)du,\quad k=1,2\text{.}
\end{equation}

Now, using (\ref{4.800}) observe that:
\begin{align*}
\text{A)}\quad&\|D\phi_n(g)(w)\|\nonumber\\
\leq& \int_{0}^\infty e^{-u}\mathbb{E}\|Dg(we^{-u}+\sigma(u)\mathbf{D}_n)\|du\nonumber\\
\leq& \|g\|_M\int_0^\infty \left(e^{-u}+2\|w\|^2e^{-3u}+2\mathbb{E}\|\mathbf{D}_n\|^2(e^{-u}-e^{-3u})\right)du\nonumber\\
\leq& \|g\|_{ M}\left(1+\frac{2}{3}\|w\|^2+\frac{4}{3}\mathbb{E}\|\mathbf{D}_n\|^2\right)\text{,}\nonumber\\
\text{B)}\quad&\|D^2\phi_n(g)(w)\|\nonumber\\
\leq &\int_0^\infty e^{-2u}\mathbb{E}\left\|D^2g(we^{-u}+\sigma(u)\mathbf{D}_n)\right\|du\nonumber\\
\leq& \|g\|_M\int_0^\infty e^{-2u}(1+\mathbb{E}\|we^{-u}+\sigma(u)\mathbf{D}_n\|)du\nonumber\\
\leq& \|g\|_M\left(\frac{1}{2}+\frac{\|w\|}{3}+\frac{\mathbb{E}\|\mathbf{D}_n\|}{3}\right)\nonumber,\\
\text{C)}\quad&\frac{\|D^2\phi_n(g)(w+h)-D^2\phi_n(g)(w)\|}{\|h\|}\nonumber\\
\leq&\|h\|^{-1}\left\|\mathbb{E}\int_0^\infty e^{-2u}D^2g((w+h)e^{-u}+\sigma(u)\mathbf{D}_n)-e^{-2u}D^2g(we^{-u}+\sigma(u)\mathbf{D}_n)du\right\|\nonumber\\
\leq&\sup_{w,h\in D^p}\frac{\|D^2g(w+h)-D^2g(w)\|}{\|h\|}\int_0^\infty e^{-2u}e^{-u}du\\
=&\sup_{w,h\in D^p}\frac{\|D^2(g+c)(w+h)-D^2(g+c)(w)\|}{3\|h\|}\text{,}
\end{align*}
uniformly in $g\in M$, for any constant $c$, which proves (\ref{norm_bound}).
\end{proof}
\subsection{An auxiliary result}
We now move to proving the main results of the paper.
We start with an auxiliary lemma in which we use Stein's method combined with Taylor expansions to bound the distance between $\mathbf{Y}_n$, as defined in Theorem \ref{local3} and $\mathbf{D}_n$, as defined in (\ref{d_n}). This result is of independent interest and will be used in all the proofs in this Section.
\begin{lemma}\label{lemma_aux}
Consider the setup of Theorem \ref{local3}. Let $\mathbf{D}_n$ be defined as in (\ref{d_n}) for the covariance matrix $\Sigma_n$ equal to the covariance matrix of \\$\left(X_{1,1},\dots,X_{1,p},\dots,X_{n,1},\dots,X_{n,p}\right)$. Assume that the two collections $\lbrace Z_{i,k}:i=1,\dots,n,\, k=1,\dots,p\rbrace$ and $\lbrace X_{i,k}:i=1,\dots,n,\, k=1,\dots,p\rbrace$ are independent.
Let $g\in M$, as defined in Section \ref{section 2}. Then:
\begin{align*}
&\left|\mathbb{E}g(\mathbf{Y}_n)-\mathbb{E}g(\mathbf{D}_n)\right|\\
\leq&\frac{\|g\|_{M}}{6}\sum_{i=1}^n\mathbb{E}\left\lbrace\left(\sum_{k,l,m=1}^p\left[\left(X_{i,k}\right)^2\|J_{i,k}\|^2\mathbb{1}_{[1,\lambda_k]}(i)\left(\sum_{j\in \mathbb{A}_i}X_{j,l}\|J_{j,l}\|\mathbb{1}_{[1,\lambda_l]}(j)\right)^2\right.\right.\right.\nonumber\\
&\left.\left.\left.\phantom{\frac{\|g\|_{M}}{3}\sum_{i=1}^n\mathbb{E}\lbrace}\cdot\left(\sum_{j\in \mathbb{A}_i}X_{j,m}\|J_{j,m}\|\mathbb{1}_{[1,\lambda_m]}(j)\right)^2\right]\right)^{1/2}\right\rbrace\\
&+\frac{\|g\|_M}{3}\sum_{i=1}^n\sum_{j\in \mathbb{A}_i}\sum_{k,l=1}^p\mathbb{E}\left\lbrace\vphantom{\left[\left(\sum_r\in A_j\right)^2\right]^{1/2}}\left[\vphantom{\left(\sum_r\in A_j\right)^2}\sum_{m=1}^p\left(\vphantom{\sum_r\in A_j}X_{i,k}\,\|J_{i,k}\|\,X_{j,l}\,\|J_{j,l}\|\,\mathbb{1}_{[1,\lambda_k]}(i)\mathbb{1}_{[1,\lambda_l]}(j)\right.\right.\right.\nonumber\\
&\left.\left.\left.\phantom{\frac{\|g\|_M}{3}\sum_{i=1}^n\sum_{j\in \mathbb{A}_i}\sum_{k,l=1}^p\mathbb{E}\lbrace}\cdot\sum_{r\in \mathbb{A}_{ij}\,\cap\, \mathbb{A}_i^c}X_{r,m}\|J_{r,m}\|\mathbb{1}_{[1,\lambda_m]}(r)\right)^2\right]^{1/2}\right\rbrace\\
&+\frac{\|g\|_{M}}{3}\sum_{i=1}^n\sum_{j\in\mathbb{A}_i}\sum_{k,l=1}^p\left\lbrace\vphantom{\left[\sqrt{\left(\sum_1^p\right)^2}\right]}\left|\mathbb{E}\left[X_{i,k}X_{j,l}\right]\right|\mathbb{1}_{[1,\lambda_k]}(i)\mathbb{1}_{[1,\lambda_l]}(j)\right.\nonumber\\
&\phantom{\frac{\|g\|_{M}}{3}\sum_{i=1}^n\sum_{j\in\mathbb{A}_i}\sum_{k,l=1}^p\lbrace}\cdot\left.\mathbb{E}\left[\|J_{i,k}\|\,\|J_{j,l}\|\sqrt{\sum_{m=1}^p\left(\sum_{r\in \mathbb{A}_{ij}}X_{r,m}\|J_{r,m}\|\mathbb{1}_{[1,\lambda_m]}(r)\right)^2}\right]\right\rbrace.
\end{align*}
\end{lemma}
The proof of Lemma \ref{lemma_aux} is based on manipulating the Stein operator, given in Proposition \ref{prop12.7}, using Taylor's theorem. 

\begin{proof}[Proof of Lemma \ref{lemma_aux}]
Let $g_n=g-\mathbb{E}g(\mathbf{D}_n)$ and $f_n=\phi_n(g_n)$, as defined in (\ref{phi}). From Proposition \ref{prop12.7} we know that:
$$\left|\mathbb{E}g(\mathbf{Y}_n)-\mathbb{E}g(\mathbf{D}_n)\right|=\left|\mathbb{E}\left[Df_n(\mathbf{Y}_n)\left[\mathbf{Y}_n\right]-D^2f_n(\mathbf{Y}_n)\left[\mathbf{D}_n,\mathbf{D}_n\right]\right]\right|.$$

Let 
$$\mathbf{Y}_n^j=\sum_{k\in \mathbb{A}_j^c}\left(X_{k,1}\mathbb{1}_{[1,\lambda_1]}(k)J_{k,1},\dots,X_{k,p}\mathbb{1}_{[1,\lambda_p]}(k)J_{k,p}\right)$$ and 
$$\mathbf{Y}_n^{ij}=\sum_{k\in \mathbb{A}_{ij}^c}\left(X_{k,1}\mathbb{1}_{[1,\lambda_1]}(k)J_{k,1},\dots,X_{k,p}\mathbb{1}_{[1,\lambda_p]}(k)J_{k,p}\right).$$ Hence, $\mathbf{Y}_n^j$ is independent of $X_j$ for all $j$ and $\mathbf{Y}_n^{ij}$ is independent of $(X_i,X_j)$ for all $i,j$. Therefore 
$$\mathbb{E}Df_n(\mathbf{Y}_n^i)\left[\left(X_{i,1}\mathbb{1}_{[1,\lambda_1]}(i)J_{i,1,},\dots, X_{i,p}\mathbb{1}_{[1,\lambda_p]}(i)J_{i,p}\right)\right]=0.$$
For $\lbrace e_k:k=1,\dots,p\rbrace$ denoting the elements of the canonical basis of $\mathbb{R}^p$ and for $i\in\lbrace 1,\dots,n\rbrace$, we have the following identities and inequalities (note that inequality $(\ast)$ follows from Taylor's theorem): 
\begin{align}
&\left|\vphantom{\sum_j}\mathbb{E}Df_n(\mathbf{Y}_n)\left[\left(X_{i,1}\mathbb{1}_{[1,\lambda_1]}(i)J_{i,1},\dots, X_{i,p}\mathbb{1}_{[1,\lambda_p]}(i)J_{i,p}\right)\right]\right.\nonumber\\
&\left.-\mathbb{E}\left[\sum_{j\in \mathbb{A}_i}\sum_{k,l=1}^p\left(X_{i,k}\mathbb{1}_{[1,\lambda_k]}(i)\right)\left(X_{j,l}\mathbb{1}_{[1,\lambda_l]}(j)\right)D^2f_n(\mathbf{Y}_n^i)\left[e_kJ_{i,k},e_lJ_{j,l}\right]\right]\right|\nonumber\\
=&\left|\vphantom{\sum_j}\mathbb{E}Df_n(\mathbf{Y}_n)\left[\left(X_{i,1}\mathbb{1}_{[1,\lambda_1]}(i)J_{i,1},\dots, X_{i,p}\mathbb{1}_{[1,\lambda_p]}(i)J_{i,p}\right)\right]\right.\nonumber\\
&-\mathbb{E}Df_n(\mathbf{Y}_n^i)\left[\left(X_{i,1}\mathbb{1}_{[1,\lambda_1]}(i)J_{i,1},\dots, X_{i,p}\mathbb{1}_{[1,\lambda_p]}(i)J_{i,p}\right)\right]\nonumber\\
&-\mathbb{E}D^2f_n(\mathbf{Y}_n^i)\left[\vphantom{\sum_j}\left(X_{i,1}\mathbb{1}_{[1,\lambda_1]}(i)J_{i,1},\dots, X_{i,p}\mathbb{1}_{[1,\lambda_p]}(i)J_{i,p}\right),\right.\nonumber\\
&\phantom{-\mathbb{E}Df_n(\mathbf{Y}_n^i)[}\left.\left.\sum_{j\in \mathbb{A}_i}\left(X_{j,1}\mathbb{1}_{[1,\lambda_1]}(j)J_{j,1},\dots X_{j,p}\mathbb{1}_{[1,\lambda_p]}(j)J_{j,p}\right)\right]\right|\nonumber\\
\stackrel{(\ast)}\leq&\frac{1}{2}\sup_{w,h\in D^p}\frac{\left\|D^2f_n(w+h)-D^2f_n(w)\right\|}{\|h\|}\nonumber\\
&\cdot\mathbb{E}\left[\vphantom{\left\|\sum_j\right\|^2}\left\|\left(X_{i,1}\mathbb{1}_{[1,\lambda_1]}(i)J_{i,1},\dots, X_{i,p}\mathbb{1}_{[1,\lambda_p]}(i)J_{i,p}\right)\right\|\right.\nonumber\\
&\left.\cdot\left\|\sum_{j\in \mathbb{A}_i}\left(X_{j,1}\mathbb{1}_{[1,\lambda_1]}(j)J_{j,1},\dots,X_{j,p}\mathbb{1}_{[1,\lambda_p]}(j)J_{j,p}\right)\right\|\|Y_n-Y_n^i\|\right]\nonumber\\
\stackrel{(\ref{norm_bound})C)}\leq&\frac{\|g\|_{M}}{6}\mathbb{E}\left[\vphantom{\left\|\sum_j\right\|^2}\left\|\left(X_{i,1}\mathbb{1}_{[1,\lambda_1]}(i)J_{i,1},\dots, X_{i,p}\mathbb{1}_{[1,\lambda_p]}(i)J_{i,p}\right)\right\|\right.\nonumber\\
&\left.\cdot\left\|\sum_{j\in \mathbb{A}_i}\left(X_{j,1}\mathbb{1}_{[1,\lambda_1]}(j)J_{j,1},\dots,X_{j,p}\mathbb{1}_{[1,\lambda_p]}(j)J_{j,p}\right)\right\|\|Y_n-Y_n^i\|\right]\nonumber\\
=&\frac{\|g\|_{M}}{6}\mathbb{E}\left[\vphantom{\left\|\sum_j\right\|^2}\left\|\left(X_{i,1}\mathbb{1}_{[1,\lambda_1]}(i)J_{i,1},\dots, X_{i,p}\mathbb{1}_{[1,\lambda_p]}(i)J_{i,p}\right)\right\|\right.\nonumber\\
&\left.\cdot\left\|\sum_{j\in \mathbb{A}_i}\left(X_{j,1}\mathbb{1}_{[1,\lambda_1]}(j)J_{j,1},\dots,X_{j,p}\mathbb{1}_{[1,\lambda_p]}(j)J_{j,p}\right)\right\|^2\right]\nonumber\\
\leq& \frac{\|g\|_{M}}{6}\mathbb{E}\left\lbrace\left(\sum_{k,l,m=1}^p\left[\left(X_{i,k}\right)^2\|J_{i,k}\|^2\mathbb{1}_{[1,\lambda_k]}(i)\left(\sum_{j\in \mathbb{A}_i}X_{j,l}\|J_{j,l}\|\mathbb{1}_{[1,\lambda_l]}(j)\right)^2\right.\right.\right.\nonumber\\
&\left.\left.\left.\cdot\left(\sum_{j\in \mathbb{A}_i}X_{j,m}\|J_{j,m}\|\mathbb{1}_{[1,\lambda_m]}(j)\right)^2\right]\right)^{1/2}\right\rbrace.\label{12.15}
\end{align}

Furthermore, for all $i,j\in\lbrace 1,\dots, n\rbrace$,
\begin{align}
&\left|\mathbb{E}\left[X_{i,k}\mathbb{1}_{[1,\lambda_k]}(i)X_{j,l}\mathbb{1}_{[1,\lambda_l]}(j)D^2f_n(\mathbf{Y}_n^i)\left[e_kJ_{i,k},e_lJ_{j,l}\right]\right]\right.\nonumber\\
&\left.-\mathbb{E}\left[X_{i,k}\mathbb{1}_{[1,\lambda_k]}(i)X_{j,l}\mathbb{1}_{[1,\lambda_l]}(j)D^2f_n(\mathbf{Y}_n^{i,j})\left[e_kJ_{i,k},e_lJ_{j,l}\right]\right]\right|\nonumber\\
\stackrel{(\ref{norm_bound})C)}\leq& \frac{\|g\|_{M}}{3}\mathbb{E}\left\lbrace\vphantom{\left[\left(\sum_r\in A_j\right)^2\right]^{1/2}}\left[\vphantom{\left(\sum_r\in A_j\right)^2}\sum_{m=1}^p\left(\vphantom{\sum_r\in A_j}X_{i,k}\,\|J_{i,k}\|\,X_{j,l}\,\|J_{j,l}\|\,\mathbb{1}_{[1,\lambda_k]}(i)\mathbb{1}_{[1,\lambda_l]}(j)\right.\right.\right.\nonumber\\
&\left.\left.\left.\cdot\sum_{r\in \mathbb{A}_{ij}\,\cap\, \mathbb{A}_i^c}X_{r,m}\|J_{r,m}\|\mathbb{1}_{[1,\lambda_m]}(r)\right)^2\right]^{1/2}\right\rbrace\label{12.16}
\end{align}
and
\begin{align}
&\left|\mathbb{E}\left[X_{i,k}\mathbb{1}_{[1,\lambda_k]}(i)X_{j,l}\mathbb{1}_{[1,\lambda_l]}(j)D^2f_n\left(\mathbf{Y}_n^{i,j}\right)\left[e_kJ_{i,k},e_lJ_{j,l}\right]\right]\right.\nonumber\\
&\left.-\mathbb{E}\left[X_{i,k}X_{j,l}\right]\mathbb{1}_{[1,\lambda_k]}(i)\mathbb{1}_{[1,\lambda_l]}(j)\mathbb{E}\left[D^2f_n(\mathbf{Y}_n)\left[e_kJ_{i,k},e_lJ_{j,l}\right]\right]\right|\nonumber\\
=&\left|\mathbb{E}\left[X_{i,k}X_{j,l}\right]\mathbb{1}_{[1,\lambda_k]}(i)\mathbb{1}_{[1,\lambda_l]}(j)\mathbb{E}\left[\left(D^2f_n(\mathbf{Y}_n)-D^2f_n\left(\mathbf{Y}_n^{i,j}\right)\right)\left[e_kJ_{i,k},e_lJ_{j,l}\right]\right]\right|\nonumber\\
\stackrel{(\ref{norm_bound})C)}\leq&\frac{\|g\|_{M}}{3}\left|\mathbb{E}\left[X_{i,k}X_{j,l}\right]\right|\mathbb{1}_{[1,\lambda_k]}(i)\mathbb{1}_{[1,\lambda_l]}(j)\nonumber\\
&\cdot\mathbb{E}\left[\|J_{i,k}\|\,\|J_{j,l}\|\sqrt{\sum_{m=1}^p\left(\sum_{r\in \mathbb{A}_{ij}}X_{r,m}\|J_{r,m}\|\mathbb{1}_{[1,\lambda_m]}(r)\right)^2}\right].\label{12.17}
\end{align}
Summing (\ref{12.15}) over $i=1,\dots, n$ and (\ref{12.16}) and (\ref{12.17}) over $i=1,\dots, n$, $j\in \mathbb{A}_i$ and $k,l=1,\dots, p$ will give us a bound on $\left|\mathbb{E}\mathcal{A}_ng(\mathbf{Y}_n)\right|$, as defined in Proposition \ref{prop12.7}, i.e. a bound on $\left|\mathbb{E}g(\mathbf{Y}_n)-\mathbb{E}g(\mathbf{D}_n)\right|$.
\end{proof}
\subsection{Proof of Theorem \ref{local3}}
In the proof of Theorem \ref{local3} below, we will use auxiliary processes $\tilde{\mathbf{D}}_n$ and $\tilde{\mathbf{A}}_n$. In order to define them, we let $\left(\tilde{Z}_{1,1},\dots, \tilde{Z}_{1,p},\tilde{Z}_{2,1},\dots,\tilde{Z}_{2,p},\dots, \tilde{Z}_{n,1},\dots, \tilde{Z}_{n,p}\right)$ be a centred Gaussian vector with the same covariance as that of \\$\left(X_{1,1},\dots,X_{1,p},\dots,X_{n,1},\dots,X_{n,p}\right)$ and independent of \\$\left(X_{1,1},\dots,X_{1,p},\dots,X_{n,1},\dots,X_{n,p}\right)$. We also let $\left\lbrace\left(Z_{i,1},\dots,Z_{i,p}\right):\,i=1,\dots,n\right\rbrace$ be a collection of i.i.d. Gaussian vectors with mean zero and covariance $\Sigma$, independent of the collections $\lbrace J_{i,k}:\,i=1,\dots,n,\,k=1,\dots,p\rbrace$ and $\lbrace X_{i,k}:\,i=1,\dots,n,\,k=1,\dots,p\rbrace$. The auxiliary processes are defined for $t\in[0,1]$ in the following way:
\begin{align}
&\tilde{\mathbf{D}}_n(t)=\left(\sum_{i=1}^{ \lambda_1}\tilde{Z}_{i,1}\mathbb{1}_{[i/\lambda_1,1]}(t),\dots,\sum_{i=1}^{\lambda_p}\tilde{Z}_{i,p}\mathbb{1}_{[i/\lambda_p,1]}(t)\right);\label{tilde_d_n}\\
&\tilde{\mathbf{A}}_n(t)=\left(\frac{1}{\sqrt{\lambda_1}}\sum_{i=1}^{\lambda_1}Z_{i,1}\mathbb{1}_{[i/\lambda_1,1]}(t),\dots,\frac{1}{\sqrt{\lambda_p}}\sum_{i=1}^{\lambda_p}Z_{i,p}\mathbb{1}_{[i/\lambda_p,1]}(t)\right).\label{tilde_a_n}
\end{align}

 \textbf{Step 1} of the proof below makes a straightforward use of the mean value theorem to bound the distance between $\mathbf{D}_n$, as defined by (\ref{d_n}) and $\tilde{\mathbf{D}}_n$. In \textbf{Step 2} the distance between $\tilde{\mathbf{D}}_n$ and $\tilde{\mathbf{A}}_n$ is bounded using bounds on the distance between two multivariate Gaussian distributions (\cite[Proposition 2.8]{reinert_roellin}). In \textbf{Step 3} we couple $\tilde{\mathbf{A}}_n$ and $\mathbf{Z}$ in order to obtain a bound on $\mathbb{E}\|\tilde{\mathbf{A}}_n-\mathbf{Z}\|$ and then apply the mean value theorem again to bound $|\mathbb{E}g(\tilde{\mathbf{A}}_n)-\mathbb{E}g(\mathbf{Z})|$ for all $g\in M^1$. Those three steps combined with Lemma \ref{lemma_aux} yield the assertion. In short:
\begin{align*}
\left|\mathbb{E}g(\mathbf{Y}_n)-\mathbb{E}g(\mathbf{Z})\right|\leq& \underbrace{\left|\mathbb{E}g(\mathbf{Y}_n)-\mathbb{E}g(\mathbf{D}_n)\right|}_{\text{Lemma \ref{lemma_aux}}}+\underbrace{\left|\mathbb{E}g(\mathbf{D}_n)-\mathbb{E}g(\tilde{\mathbf{D}}_n)\right|}_{\text{Step 1}}\\
&+\underbrace{\left|\mathbb{E}g(\tilde{\mathbf{D}}_n)-\mathbb{E}g(\tilde{\mathbf{A}}_n)\right|}_{\text{Step 2}}+\underbrace{\left|\mathbb{E}g(\tilde{\mathbf{A}}_n)-\mathbb{E}g(\mathbf{Z})\right|}_{\text{Step 3}}.
\end{align*}
\begin{proof}[Proof of theorem \ref{local3}]~\\

\textbf{Step 1.} Note that, for $\mathbf{D}_n$ of Lemma \ref{lemma_aux} and $\tilde{\mathbf{D}}_n$ of (\ref{tilde_d_n}),
\begin{align}
&\left|\mathbb{E}g(\mathbf{D}_n)-\mathbb{E}g(\tilde{\mathbf{D}}_n)\right|\nonumber\\
\leq &\|g\|_{M^1}\mathbb{E}\|\mathbf{D}_n-\tilde{\mathbf{D}}_n\|\nonumber\\
\leq&\|g\|_{M^1}\mathbb{E}\left\lbrace\sup_{t\in[0,1]}\sqrt{\sum_{k=1}^p\left[\sum_{i=1}^{\lambda_k}\tilde{Z}_{i,k}\left(J_{i,k}(t)-\mathbb{1}_{[i/\lambda_k,1]}(t)\right)\right]^2}\right\rbrace\nonumber \\
\leq& \|g\|_{M^1}\sum_{k=1}^p\sum_{i=1}^{\lambda_k}\mathbb{E}|\tilde{Z}_{i,k}|\mathbb{E}\left\|J_{i,k}-\mathbb{1}_{[i/\lambda_k,1]}\right\|\nonumber\\
\leq&\|g\|_{M^1}\sum_{k=1}^p\sum_{i=1}^{\lambda_k}\sqrt{\mathbb{E}\left[(X_{i,k})^2\right]}\mathbb{E}\left\|J_{i,k}-\mathbb{1}_{[i/\lambda_k,1]}\right\|
\label{5.111},
\end{align}
giving $\epsilon_7$.

\textbf{Step 2.} Let $\lambda=\sum_{k=1}^p\lambda_k$ and consider function $f:\mathbb{R}^{\lambda}\to D^p[0,1]$ given by:
$$f\left(x_{1,1},\dots,x_{\lambda_1,1},\dots,x_{1,p},\dots,x_{\lambda_p,p}\right)=\left(\sum_{i=1}^{ \lambda_1}x_{i,1}\mathbb{1}_{[i/\lambda_1,1]},\dots,\sum_{i=1}^{\lambda_p}x_{i,p}\mathbb{1}_{[i/\lambda_p,1]}\right).$$
This function is twice differentiable with:
\begin{align*}
\text{A)}\quad &Df(x)[(h_{1,1},\dots,h_{\lambda_1,1},\dots,h_{1,p},\dots,h_{\lambda_p,p})]&\\
=&\left(\sum_{i=1}^{\lambda_1}h_{i,1}\mathbb{1}_{[i/\lambda_1,1]},\dots,\sum_{i=1}^{\lambda_p}h_{i,p}\mathbb{1}_{[i/\lambda_p,1]}\right)&\\
\text{B)}\quad &D^2f(x)[h^{(1)},h^{(2)}]=0&
\end{align*}
for all $x,h=(h_{1,1},\dots,h_{\lambda_1,1},\dots,h_{1,p},\dots,h_{\lambda_p,p}),h^{(1)},h^{(2)}\in\mathbb{R}^{np}$. We notice that for the canonical basis vectors $e_i,e_j\in\mathbb{R}^{np}$ we have:
$$\left|D^2(g\circ f)(x)[e_i,e_j]\right|=\left|D^2g(f(x))[Df(x)[e_i],Df(x)[e_j]]\right|\leq \sup_{w\in D}\|D^2g(w)\|$$
for all $x\in\mathbb{R}^{np}$. This follows from the fact that $|Df(x)[e_i]|=1$. Therefore, we can apply \cite[Proposition 2.8]{reinert_roellin} to the function $g\circ f$ and, recalling the definitions of $\tilde{\mathbf{D}}_n$ in (\ref{tilde_d_n}) and $\tilde{\mathbf{A}}_n$ in (\ref{tilde_a_n}), obtain
\begin{align}
&|\mathbb{E}g(\tilde{\mathbf{A}}_n)-\mathbb{E}g(\tilde{\mathbf{D}}_n)|\nonumber\\
\leq& \frac{1}{2}\|g\|_{M^1}\sum_{k,l=1}^p\left[\sum_{i=1}^{\lambda_k}\sum_{j\in \mathbb{A}_i\setminus\lbrace i\rbrace}\left|\mathbb{E}\left[X_{i,k}X_{j,l}\right]\right|+\sum_{i=1}^{\lambda_k\wedge\lambda_l}\left|\frac{\Sigma_{k,l}}{\sqrt{\lambda_k\lambda_l}}-\mathbb{E}[X_{i,k}X_{i,l}]\right|\right],\label{5.11}
\end{align}
giving $\epsilon_4+\epsilon_5$.

\textbf{Step 3.} We now realise a $p$-dimensional Brownian motion $\mathbf{B}$ and let $\mathbf{Z}=\Sigma^{1/2}\mathbf{B}$. We also let $$\tilde{\mathbf{A}}_n^{(j)}(t)=\mathbf{Z}^{(j)}\left(l/\lambda_j\right),\quad\text{for } t\in\left[l/\lambda_j,(l+1)/\lambda_j\right)$$ 
for every $j=1,\dots,p$, which agrees in distribution with our original definition (\ref{tilde_a_n}) of $\tilde{\mathbf{A}}_n=\left(\tilde{\mathbf{A}}_n^{(1)},\dots,\tilde{\mathbf{A}}_n^{(p)}\right)$. Now, note that, using Jensen's inequality , we have:
\begin{align*}
\mathbb{E}\|\tilde{\mathbf{A}}_n-\mathbf{Z}\|\leq& \left(\sum_{i=1}^p\mathbb{E}\left\|\tilde{\mathbf{A}}_n^{(i)}-\mathbf{Z}^{(i)}\right\|^2\right)^{1/2}\\
=&\sqrt{\sum_{i=1}^p\mathbb{E}\sup_{t\in[0,1]}\left|\mathbf{Z}^{(i)}(t)-\mathbf{Z}^{(i)}\left(\frac{\lfloor \lambda_it\rfloor}{\lambda_i}\right)\right|^2}\\
\leq&\left(\sum_{i=1}^p\Sigma_{i,i}\right)^{1/2}\sqrt{\sum_{i=1}^p\mathbb{E}\sup_{t\in[0,1]}\left|\mathbf{B}^{(i)}(t)-\mathbf{B}^{(i)}\left(\frac{\lfloor \lambda_it\rfloor}{\lambda_i}\right)\right|^2}\\
\leq&\frac{6\sqrt{5}}{\sqrt{2\log 2}}\left(\sqrt{\sum_{i=1}^p \frac{\log\left(2\lambda_i\right)}{\lambda_i}}\right)\left(\sum_{i=1}^p\Sigma_{i,i}\right)^{1/2},
\end{align*}
where the third inequality follows because $\left\|\Sigma^{1/2}\right\|_2=\sqrt{\lambda_{\text{max}}\left(\Sigma\right)}\leq\left(\sum_{i=1}^p\Sigma_{i,i}\right)^{1/2}$, where $\lambda_{\text{max}}\left(\Sigma\right)$ denotes the largest eigenvalue of $\Sigma$ and the last inequality follows by \cite[Lemma 3]{ito_processes}. Therefore:
\begin{align}
|\mathbb{E}g(\tilde{\mathbf{A}}_n)-\mathbb{E}g(\mathbf{Z})|\stackrel{\text{MVT}}\leq& \sup_{w\in D^p}\|Dg(w)\|\mathbb{E}\|\mathbf{Z}-\tilde{\mathbf{A}}_n\|\nonumber\\
\leq& \|g\|_{M^1}\frac{6\sqrt{5}}{\sqrt{2\log 2}}\left(\sqrt{\sum_{i=1}^p \frac{\log\left(2\lambda_i\right)}{\lambda_i}}\right)\left(\sum_{i=1}^p\Sigma_{i,i}\right)^{1/2},\label{5.12}
\end{align}
giving $\epsilon_6$.

Now, Lemma \ref{lemma_aux} (which gives $\epsilon_1+\epsilon_2+\epsilon_3$), combined with (\ref{5.11}), (\ref{5.111}), (\ref{5.12}), yields the assertion.
\end{proof}

\subsection{Proof of Proposition \ref{local1}}
The proof of Proposition \ref{local1} below is similar to that of Lemma \ref{lemma_aux} and \textbf{Step 3} of the proof of Theorem \ref{local3}. Due to the independence of the summands, the bound on the distance between $\mathbf{Y}_n$ and the pre-limiting Gaussian process has a simpler form than the one appearing in Theorem \ref{local3}. We now work with all $g\in M$, contrary to what is done in the proof of Theorem \ref{local3}. Hence, we need to bound both the first and second moment of the supremum distance between the pre-limiting process and the correlated Brownian motion. This is necessary for the mean value theorem to be applied in the final step.
\begin{proof}[Proof of Proposition \ref{local1}]
Let $\mathbf{D}_n$ be as in (\ref{d_n}) with $\Sigma_n$ such that the vectors $\left(\tilde{Z}_i\right)_{i=1}^n$ are i.i.d with the same covariance structure as that of $\left(X_i\right)_{i=1}^n$ and for all $i=1,\dots,n$ and $k=1,\dots,p$, $J_{i,k}=\mathbb{1}_{[i/n,1]}$. Let $g\in M$, $g_n=g-\mathbb{E}[g(\mathbf{D}_n)]$, $f_n=\phi_n(g_n)$, as in (\ref{phi}). 

Note that for $\mathbf{Y}_n^j=\mathbf{Y}_n-\frac{1}{\sqrt{n}}X_j\mathbb{1}_{[j/n,1]}$, $j=1,\dots, n$,  $\mathbf{Y}_n^{j}$ is independent of $X_j$ and
\begin{align}
&\left|n^{-1/2}\mathbb{E}Df_n(\mathbf{Y}_n)\left[X_j\mathbb{1}_{[j/n,1]}\right]-n^{-1}\sum_{k,l=1}^p\Sigma_{k,l}\mathbb{E}D^2f_n(\mathbf{Y}_n^j)\left[e_k\mathbb{1}_{[j/n,1]},e_l\mathbb{1}_{[j/n,1]}\right]\right|\nonumber\\
=&\left|n^{-1/2}\mathbb{E}Df_n(\mathbf{Y}_n)\left[X_j\mathbb{1}_{[j/n,1]}\right]-n^{-1/2}\mathbb{E}Df_n(\mathbf{Y}_n^j)\left[X_j\mathbb{1}_{[j/n,1]}\right]\right.\nonumber\\
&\left.-n^{-1}\mathbb{E}D^2f_n(\mathbf{Y}_n^j)\left[X_j\mathbb{1}_{[j/n,1]},X_j\mathbb{1}_{[j/n,1]}\right]\vphantom{n^{-1/2}}\right|\nonumber\\
\leq&\frac{n^{-3/2}}{2}\sup_{w,h\in D^p}\frac{\left\|D^2f_n(w+h)-D^2f_n(w)\right\|}{\|h\|}\mathbb{E}\left\|X_j\mathbb{1}_{[j/n,1]}\right\|^3\nonumber\\
\leq& n^{-3/2}\frac{\|g\|_M}{6}\mathbb{E}\left\|X_j\mathbb{1}_{[j/n,1]}\right\|^3\nonumber\\
=&n^{-3/2}\frac{\|g\|_M}{6}\mathbb{E}\left[\left(\left(X_j^{(1)}\right)^2+\dots+\left(X_j^{(p)}\right)^2\right)^{3/2}\right]\nonumber\\
\leq &p^{1/2}n^{-3/2}\frac{\|g\|_M}{6}\sum_{m=1}^p\mathbb{E}\left|X_j^{(m)}\right|^3,\label{12.4}
\end{align}
where the first inequality follows by Taylor's theorem and the second one by (\ref{norm_bound})C).
Also, by (\ref{norm_bound})C):
\begin{align}
&\left|n^{-1}\sum_{k,l=1}^p\Sigma_{k,l}\mathbb{E}D^2f_n(\mathbf{Y}_n^j)\left[e_k\mathbb{1}_{[j/n,1]},e_l\mathbb{1}_{[j/n,1]}\right]\right.\nonumber\\
&\left.-n^{-1}\sum_{k,l=1}^p\Sigma_{k,l}\mathbb{E}D^2f_n(\mathbf{Y}_n)\left[e_k\mathbb{1}_{[j/n,1]},e_l\mathbb{1}_{[j/n,1]}\right]\right|\nonumber\\
\leq&n^{-3/2}\frac{\|g\|_M}{3}\sum_{k,l=1}^p\left|\Sigma_{k,l}\right|\left(\sum_{m=1}^p\mathbb{E}\left|X_j^{(m)}\right|^2\right)^{1/2}.\label{12.5}
\end{align}
Let us now realise a $p$-dimensional Brownian motion $\mathbf{B}$ and let $\mathbf{Z}=\Sigma^{1/2}\mathbf{B}$. We realise it in such a way that $\Sigma^{-1/2}\mathbf{D}_n(j/n)=\mathbf{B}(j/n)$ for every $j=1,...,n$, which agrees in distribution with our original definition of $\mathbf{D}_n$. Now, note that, by \cite[Lemma 3]{ito_processes} and Doob's $L^3$ inequality:
\begin{align*}
\text{A)      }\quad \mathbb{E}\|\mathbf{Z}-\mathbf{D}_n\|\leq& \sqrt{\sum_{i=1}^p\mathbb{E}\left\|\mathbf{Z}^{(i)}-\mathbf{D}_n^{(i)}\right\|^2}\leq \frac{6\sqrt{5}}{\sqrt{2\log 2}}n^{-1/2}\sqrt{\log 2n}\left(\sum_{i=1}^p\left|\Sigma_{i,i}\right|\right)^{1/2};\\
\text{B)}\quad \mathbb{E}\left\|\mathbf{Z}-\mathbf{D}_n\right\|^3\leq& p^{1/2}\sum_{i=1}^p\mathbb{E}\|\mathbf{Z}^{(i)}-\mathbf{D}_n^{(i)}\|^3\\
\leq& p^{1/2}\frac{1080}{\sqrt{\pi}(\log 2)^{3/2}}n^{-3/2}(\log 2n)^{3/2}\sum_{i=1}^p\left|\Sigma_{i,i}\right|^{3/2};\\
\text{C)              }\quad (\mathbb{E}\|\mathbf{Z}\|^3)^{2/3}\leq& \left(p^{1/2}\sum_{i=1}^p\mathbb{E}\|\mathbf{Z}^{(i)}\|^3\right)^{2/3}\leq \frac{9p^{1/3}}{2\pi^{1/3}}\left(\sum_{i=1}^p\left|\Sigma_{i,i}\right|^{3/2}\right)^{2/3}.
\end{align*}
Therefore:
\begin{align}
&|\mathbb{E}g(\mathbf{D}_n)-\mathbb{E}g(\mathbf{Z})|\nonumber\\
\stackrel{\text{MVT}}\leq&\mathbb{E}\left[\sup_{c\in[0,1]}\|Dg(\tilde{\mathbf{Z}}+c(\mathbf{D}_n-\mathbf{Z}))\|\|\mathbf{Z}-\mathbf{D}_n\|\right]\nonumber\\
\leq&\|g\|_{ M}\mathbb{E}\left[\sup_{c\in[0,1]}(1+\|\mathbf{Z}+c(\mathbf{D}_n-\mathbf{Z})\|^2)\|\mathbf{Z}-\mathbf{D}_n\|\right]\nonumber\\
\leq &\|g\|_{ M}\left\lbrace \mathbb{E}\|\mathbf{Z}-\mathbf{D}_n\|+2\mathbb{E}\|\mathbf{Z}-\mathbf{D}_n\|^3+2(\mathbb{E}\|\mathbf{Z}\|^3)^{2/3}(\mathbb{E}\|\mathbf{D}_n-\mathbf{Z}\|^3)^{1/3}\right\rbrace\nonumber\\
\leq&\|g\|_M\left\lbrace n^{-1/2}\sqrt{\log 2n}\left[\frac{6\sqrt{5}}{\sqrt{2\log 2}}\left(\sum_{i=1}^p\left|\Sigma_{i,i}\right|\right)^{1/2}+\frac{54\cdot 5^{1/3}p^{1/2}}{\sqrt{\pi\log 2}}\sum_{i=1}^p|\Sigma_{i,i}|^{3/2}\right]\right.\nonumber\\
&\left.+ n^{-3/2}(\log 2n)^{3/2}p^{1/2}\frac{2160}{\sqrt{\pi}(\log 2)^{3/2}}\sum_{i=1}^p\left|\Sigma_{i,i}\right|^{3/2}\right\rbrace.\label{12.7}
\end{align}
We now sum (\ref{12.4}) and (\ref{12.5}) and sum them over $j$, which, combined with (\ref{12.7}) yields the result.
\end{proof}
\subsection{Proof of Theorem \ref{theorem_u_stats}}
In \textbf{Step 1} of the proof of Theorem \ref{theorem_u_stats} below, we consider a scaled sum of i.i.d random variables $w(X_i)$ and apply Lemma \ref{lemma_aux} together with an argument similar to \textbf{Step 1} and \textbf{Step 3} of the proof of Theorem \ref{local3} in order to bound the distance between this scaled sum and $\mathbf{Z}$. In \textbf{Step 2} we bound the distance between this scaled sum and our original process $\mathbf{Y}_n$ by bounding the second moment of the supremum distance between them and then using the mean value theorem. 
\begin{proof}[Proof of Theorem \ref{theorem_u_stats}]
Let $g\in M^2$. 

\textbf{Step 1.} As in the proof of the invariance principle for U-statistics of \cite{hall}, we start by considering the behaviour of the following process $(\tilde{\mathbf{Y}}_n(t),t\geq 0)$:
\begin{align*}
\tilde{\mathbf{Y}}_n(t)=&\frac{n^{-3/2}}{\sigma_wt}\sum_{1\leq i_1<i_2\leq \lfloor nt\rfloor}\left(w(X_{i_1})+w(X_{i_2})\right)
=
\frac{1}
{\sqrt{n}\sigma_w}\sum_{i=1}^n w(X_i)J_{i,n}(t)\text{,}
\end{align*}
where $J_{i,n}(t)=\frac{(\lfloor nt\rfloor -1)\mathbb{1}_{[i/n,1]}(t)}{nt}$. Recall that $w(x)=\mathbb{E}h(X_1,x)$.
Let $\mathbf{A}_n(t)=n^{-1/2}\sum_{i=1}^n Z_iJ_{i,n}(t)$ and $\hat{\mathbf{A}}_n(t)=n^{-1/2}\sum_{i=1}^{\lfloor nt\rfloor} Z_i$, where $Z_i\stackrel{\text{i.i.d.}}\sim\mathcal{N}(0,1)$. 

Note that Lemma \ref{lemma_aux} readily yields that:
\begin{equation}
\left|\mathbb{E}g(\tilde{\mathbf{Y}}_n)-\mathbb{E}g(\mathbf{A}_n)\right|\leq \frac{\|g\|_M}{6\sigma_w^3}n^{-1/2}\left(\mathbb{E}\left|w(X_1)\right|^3+2\sigma_w^2\mathbb{E}|w(X_1)|\right),\label{5.13}
\end{equation}
as $\|J_{i,n}\|\leq 1$ for all $i,n\in\mathbb{N}$ and $w(X_i)$'s for $i=1,\dots, n$ are independent.

We see that, by Doob's $L^2$ inequality, we have for every $m$:
\begin{align*}
&\mathbb{E}\left[\max_{1\leq l\leq m}\left|\sum_{i=1}^lZ_i\right|\right]^2\leq 4m=4\sum_{i=1}^m 1.
\end{align*}
Therefore, using \cite[Theorem 1]{fazekas} for inequality $(\ast)$, we obtain:
\begin{align}
\text{A)}\quad\mathbb{E}\|\mathbf{A}_n-\hat{\mathbf{A}}_n\|^2\leq& n^{-1}\mathbb{E}\left[\max_{1\leq l\leq n}\left|\frac{l-1}{l+1}-1\right|\left|\sum_{i=1}^lZ_i\right|\right]^2\nonumber\\
\leq& n^{-1}4\mathbb{E}\left[\max_{1\leq l\leq n}\left|\frac{\sum_{i=1}^l Z_i}{l+1}\right|\right]^2\nonumber\\
\stackrel{(\ast)}\leq &32n^{-1}\sum_{i=1}^n\frac{1}{i^2}\nonumber\\
\leq&\frac{16\pi^2}{3}n^{-1};\nonumber\\
\text{B)}\quad\mathbb{E}\|\mathbf{A}_n-\hat{\mathbf{A}}_n\|\leq&\sqrt{\mathbb{E}\|\mathbf{A}_n-\hat{\mathbf{A}}_n\|^2}\leq \frac{4\pi}{\sqrt{3}}n^{-1/2}.\label{9.7}
\end{align}
Doob's $L^2$ inequality readily gives us:
\begin{equation}
\mathbb{E}\|\hat{\mathbf{A}}_n\|^2=\mathbb{E}\left[\left(\max_{1\leq m\leq n}n^{-1/2}\left|\sum_{i=1}^m Z_i\right|\right)^2\right]\leq 4.\label{9.77}
\end{equation}
It follows that:
\begin{align}
&|\mathbb{E}g(\mathbf{A}_n)-\mathbb{E}g(\hat{\mathbf{A}}_n)|\nonumber\\\leq &\mathbb{E}\left[\sup_{c\in[0,1]}\|Dg((1-c)\hat{\mathbf{A}}_n+c\mathbf{A}_n)\|\|\mathbf{A}_n-\hat{\mathbf{A}}_n\|\right]\nonumber \\
\leq& \|g\|_{M^2}\mathbb{E}\left[\sup_{c\in[0,1]}(1+\|\hat{\mathbf{A}}_n+c(\mathbf{A}_n-\hat{\mathbf{A}}_n)\|)\|\mathbf{A}_n-\hat{\mathbf{A}}_n\|\right]\nonumber\\
\leq& \|g\|_{M^2}\left(\mathbb{E}\|\mathbf{A}_n-\hat{\mathbf{A}}_n\|+\mathbb{E}\|\mathbf{A}_n-\hat{\mathbf{A}}_n\|^2+\sqrt{\mathbb{E}\|\hat{\mathbf{A}}_n\|^2}\sqrt{\mathbb{E}\|\mathbf{A}_n-\hat{\mathbf{A}}_n\|^2}\right)\nonumber\\
\leq &\|g\|_{M^2}\left(\frac{12\pi}{\sqrt{3}}n^{-1/2}+\frac{16\pi^2}{3}n^{-1}\right),\label{9.13}
\end{align}
where the first inequality follows from the mean value theorem and the last one follows from (\ref{9.7}) and (\ref{9.77}).
Also, by \cite[Lemma 3]{ito_processes} and Doob's $L^2$ inequality:
\begin{align*}
&\text{A)}\quad\mathbb{E}\|\hat{\mathbf{A}}_n-\mathbf{Z}\|\leq \frac{30}{\sqrt{\pi\log 2}}n^{-1/2}\sqrt{\log 2n}\\
&\text{B)}\quad\mathbb{E}\|\hat{\mathbf{A}}_n-\mathbf{Z}\|^2\leq \frac{90}{\log 2}n^{-1}\log 2n\\
&\text{C)}\quad\mathbb{E}\|\mathbf{Z}\|^2\leq4
\end{align*}
and therefore:
\begin{align}
&|\mathbb{E}g(\hat{\mathbf{A}}_n)-\mathbb{E}g(\mathbf{Z})|\nonumber\\
\leq& \|g\|_{M^2}\left(\mathbb{E}\|\hat{\mathbf{A}}_n-\mathbf{Z}\|+\mathbb{E}\|\hat{\mathbf{A}}_n-\mathbf{Z}\|^2+\sqrt{\mathbb{E}\|\mathbf{Z}\|^2}\sqrt{\mathbb{E}\|\hat{\mathbf{A}}_n-\mathbf{Z}\|^2}\right)\nonumber\\
\leq&\|g\|_{M^2}n^{-1/2}\left[\left(\frac{30}{\sqrt{\pi\log 2}}+\frac{12\sqrt{5}}{\sqrt{2\log 2}}\right)\sqrt{\log 2n}+\frac{90}{\log 2} n^{-1/2}\log 2n\right]\text{.}\label{9.12}
\end{align}
\textbf{Step 2.} We now wish to find a bound on $|\mathbb{E}g(\tilde{\mathbf{Y}}_n)-\mathbb{E}g(\mathbf{Y}_n)|$. Note that:
$$\mathbf{Y}_n-\tilde{\mathbf{Y}}_n=\frac{n^{-3/2}}{\sigma_wt}\sum_{1\leq i_1<i_2\leq \lfloor nt\rfloor} \left(h(X_{i_1},X_{i_2})-w(X_{i_1})-w(X_{i_2})\right).$$
Let  $\phi_h^2=\mathbb{E}h^2(X_1,X_2)$. First, note that, if $\mu=\mathcal{L}(X_1)$ (i.e. $\mu$ is the law of $X_1$),
\begin{align*}
&\mathbb{E}\left[\left(h(X_1,X_2)-w(X_1)-w(X_2)\right)\left(h(X_1,X_3)-w(X_1)-w(X_3)\right)\right]\\
=&\mathbb{E}\left[h(X_1,X_2)h(X_1,X_3)\right]-2\mathbb{E}\left[h(X_1,X_2)w(X_1)\right]+\mathbb{E}w^2(X_1)\\
=&\int\int\int h(x,y)h(x,z)\mu(dx)\mu(dy)\mu(dz)\\
&-2\int\int h(x,y)\int h(x,z)\mu(dz)\mu(dx)\mu(dy)\\
&+\int\int h(x,y)\mu(dy)\int h(x,z)\mu(dz)\mu(dx)=0,
\end{align*}
where the first equality follows by the fact that $w(X_2)$ is independent of $h(X_1,X_3)$, $w(X_1)$ and $w(X_3)$, $w(X_3)$ is independent of $h(X_1,X_2)$, $w(X_1)$ and $w(X_2)$, and $\mathbb{E}w(X_2)=\mathbb{E}w(X_3)=0$.
Therefore:
\begin{align}
&\mathbb{E}\left[\sum_{1\leq i_1<i_2\leq m} \left(h(X_{i_1},X_{i_2})-w(X_{i_1})-w(X_{i_2})\right)\right]^2\nonumber\\
=&{m\choose 2}\mathbb{E}\left[h(X_1,X_2)-w(X_1)-w(X_2)\right]^2\nonumber\\
=&{m\choose 2} \left[\sigma_h^2+2\sigma_w^2-4\int\int h(x,y)\int h(x,z)\mu(dz)\mu(dx)\mu(dy)\right]\nonumber\\
=&{m\choose 2}(\sigma_h^2-2\sigma_w^2).\label{9.11}
\end{align}
Now, $\sum_{1\leq i_1<i_2\leq m} \left(h(X_{i_1},X_{i_2})-w(X_{i_1})-w(X_{i_2})\right)$ is a martingale with respect to the filtration $\sigma(X_1,...,X_m)$. Indeed:
\begin{align*}
&\mathbb{E}\left[\left.\sum_{1\leq i_1<i_2\leq m+1} \left(h(X_{i_1},X_{i_2})-w(X_{i_1})-w(X_{i_2})\right)\right|X_1,...,X_m\right]\\
=&\sum_{1\leq i_1<i_2\leq m} \left(h(X_{i_1},X_{i_2})-w(X_{i_1})-w(X_{i_2})\right)\\
&+\mathbb{E}\left[\left.\sum_{i=1}^{m} \left(h(X_{i},X_{m+1})-w(X_{i})-w(X_{m+1})\right)\right|X_1,...,X_m\right]\\
=&\sum_{1\leq i_1<i_2\leq m} \left(h(X_{i_1},X_{i_2})-w(X_{i_1})-w(X_{i_2})\right)+\sum_{i=1}^m\left(\mathbb{E}\left[\left.h(X_i,X_{m+1})\right|X_i\right]-w(X_i)\right)\\
=&\sum_{1\leq i_1<i_2\leq m} \left(h(X_{i_1},X_{i_2})-w(X_{i_1})-w(X_{i_2})\right).
\end{align*}
Hence, Doob's inequalities give us, for every $m$, such that $1\leq m\leq n$:
\begin{equation*}
\begin{aligned}
\mathbb{E}\left[\max_{1\leq l\leq m}\left|\sum_{1\leq i_1<i_2\leq l} \left(h(X_{i_1},X_{i_2})-w(X_{i_1})-w(X_{i_2})\right)\right|\right]^2\stackrel{(\ref{9.11})}\leq 4{m\choose 2}(\sigma_h^2-2\sigma_w^2) .
\end{aligned}
\end{equation*}
Then, by \cite[Theorem 1]{fazekas}, applied with $\beta_i=\alpha_i=i$ and $r=2$, and using the fact that ${m\choose 2}=\sum_{i=1}^m(i-1)$, we obtain:
\begin{align}
 \mathbb{E}\|\mathbf{Y}_n-\tilde{\mathbf{Y}}_n\|^2=&\frac{n^{-3}}{\sigma_w^2}\mathbb{E}\left[\sup_{t\in[0,1]}\left|t^{-1}\sum_{1\leq i_1<i_2\leq \lfloor nt\rfloor} \left(h(X_{i_1},X_{i_2})-w(X_{i_1})-w(X_{i_2})\right)\right|^2\right]\nonumber\\
=&\frac{n^{-1}}{\sigma_w^2}\mathbb{E}\left[\max_{1\leq l\leq n}l^{-1}\left|\sum_{1\leq i_1<i_2\leq l} \left(h(X_{i_1},X_{i_2})-w(X_{i_1})-w(X_{i_2})\right)\right|\right]^2\nonumber \\
\leq &16\left(\frac{\sigma_h^2}{\sigma_w^2}-2\right)\sum_{i=1}^n\frac{1}{i}n^{-1}\leq 16\left(\frac{\sigma_h^2}{\sigma_w^2}-2\right)n^{-1}\log 3n\text{.}\label{9.15}
\end{align}
Also, by Doob's $L^2$ inequality: 
\begin{align}
\mathbb{E}\|\tilde{\mathbf{Y}}_n\|^2=&n^{-3}\mathbb{E}\left[\sup_{t\in[0,1]}\left|\frac{\lfloor nt\rfloor -1}{t}\sum_{i=1}^{\lfloor nt\rfloor}\frac{w(X_i)}{\sigma_w}\right|\right]^2=n^{-1}\mathbb{E}\left[\sup_{1\leq l\leq n}\left|\frac{l-1}{l}\sum_{i=1}^{l}\frac{w(X_i)}{\sigma_w}\right|\right]^2 \leq 4.\label{9.16}
\end{align}
Therefore:
\begin{align}
&|\mathbb{E}g(\mathbf{Y}_n)-\mathbb{E}g(\tilde{\mathbf{Y}}_n)|\nonumber\\
\leq &\mathbb{E}\left[\sup_{c\in[0,1]}\|Dg\left((1-c)\tilde{\mathbf{Y}}_n+c\mathbf{Y}_n\right)\|\|\mathbf{Y}_n-\tilde{\mathbf{Y}}_n\|\right] \nonumber\\
\leq &\|g\|_{M^2}\mathbb{E}\left[\sup_{c\in[0,1]}(1+\|\tilde{\mathbf{Y}}_n+c(\mathbf{Y}_n-\hat{\mathbf{A}}_n)\|)\|\mathbf{Y}_n-\tilde{\mathbf{Y}}_n\|\right]\nonumber\\
\leq &\|g\|_{M^2}\left(\mathbb{E}\|\mathbf{Y}_n-\tilde{\mathbf{Y}}_n\|+\mathbb{E}\|\mathbf{Y}_n-\tilde{\mathbf{Y}}_n\|^2+\sqrt{\mathbb{E}\|\tilde{\mathbf{Y}}_n\|^2}\sqrt{\mathbb{E}\|\mathbf{Y}_n-\tilde{\mathbf{Y}}_n\|^2}\right)\nonumber\\
\leq& \|g\|_{M^2} \left(12\left(\frac{\sigma_h^2}{\sigma_w^2}-2\right)^{1/2}n^{-1/2}\sqrt{\log 3n}+16\left(\frac{\sigma_h^2}{\sigma_w^2}-2\right)n^{-1}\log 3n\right),\label{9.17}
\end{align}
where the first inequality follows from the mean value theorem and the last one follows by (\ref{9.15}) and (\ref{9.16}).

We combine (\ref{5.13}), (\ref{9.13}), (\ref{9.12}) and (\ref{9.17}) to obtain the assertion.
\end{proof}
\begin{remark}
While, in the proof of Theorem \ref{theorem_u_stats} above, it is possible to obtain a bound on $|\mathbb{E}g(\tilde{\mathbf{Y}}_n)-\mathbb{E}g(\mathbf{Z})|$ for any $g\in M$, using methods analogous to those which let us prove Theorem \ref{local3}, the situation becomes more complicated when it comes to approximating the remainder. This is because using Doob's $L^3$ inequality and \cite[Corollary 1]{fazekas} for $\mathbb{E}\|\mathbf{Y}_n-\tilde{\mathbf{Y}}_n\|^3$ gives a bound which does not converge to $0$ with $n$. Therefore, in (\ref{9.17}) we cannot go beyond the second moment of $\|\mathbf{Y}_n-\tilde{\mathbf{Y}}_n\|$. Hence, for our technique of proof, it is necessary that we assume $g\in M^2$, as defined by (\ref{m_2}).
\end{remark}

\begin{remark}
The stronger assumption of $g\in M^1$ in Theorem \ref{theorem_u_stats} would simplify its proof. Namely, using the notation of the proof of Theorem \ref{theorem_u_stats}, we could treat $\tilde{Y}_n$ as a scaled sum of i.i.d. mean zero, variance $1$ random variables $\frac{w(X_i)}{\sigma_w}$. Using (\ref{9.7}) and applying Theorem \ref{local3} gives:
$$\left|\mathbb{E}g(\tilde{\mathbf{Y}}_n)-\mathbb{E}g(\mathbf{Z})\right|\leq\frac{\|g\|_{M^1}}{2}n^{-1/2}\left(\frac{\mathbb{E}|w(X_1)|^3}{\sigma_w^3}+8+10\sqrt{\log 2n}\right)$$
and (\ref{9.17}) could be substituted with:
$$\left|\mathbb{E}g(\mathbf{Y}_n)-\mathbb{E}g(\tilde{\mathbf{Y}}_n)\right|\leq \|g\|_{M^1}\mathbb{E}\|\mathbf{Y}_n-\tilde{\mathbf{Y}}_n\|\stackrel{(\ref{9.15})}\leq \|g\|_M \left(\frac{\sigma_h^2}{\sigma_w^2}-2\right)^{1/2}\frac{4\sqrt{\log 3n}}{n^{1/2}}.$$
\end{remark}

\begin{appendices}
\section{Appendix: Proof of Proposition \ref{prop_m}}
As in the proof of \cite[Proposition 3.1]{functional_combinatorial}, we note that, by Skorokhod's representation theorem, $\mathbf{Z}_n$ and $\mathbf{Z}$ can be defined on the same probability space in such a way that $\|\mathbf{Z}_n-\mathbf{Z}\|\xrightarrow{n\to\infty} 0$ a.s. (as $\mathbf{Z}$ is continuous). The fact that $C([0,1],\mathbb{R}^p)$ equipped with norm $\|\cdot\|$ is separable, by the Stone-Weierstrass theorem, lets us use the argument of the proof of the Skorokhod representation theorem presented in \cite[Chapter 5]{billingsley} and conclude that it is enough to show that $\mathbb{P}[\mathbf{Y}_n\in B]\to\mathbb{P}[\mathbf{Z}\in B]$ for all sets $B=\bigcap_{1\leq l\leq L}B_l$, where $B_l=\lbrace w\in D^p:\|w-s_l\|<\gamma_l\rbrace$, $s_l\in C([0,1],\mathbb{R}^p)$ and $\gamma_l$ is such that $\mathbb{P}[\mathbf{Z}\in\partial B_l]=0$. Let us fix such a set $B$.

Let $\phi:\mathbb{R}^+\to[0,1]$ be a non-increasing, three times continuously differentiable function satisfying, $\phi(x)=1$ for $x\leq 0$ and $\phi(x)=0$ for $x\geq 1$ and  fix some $0<\epsilon,\eta_n\leq 1,p_n\geq 4$. Define $g_{l,n}:D^p\to\mathbb{R}$ by: 
\begin{equation}\label{g_ln}
g_{l,n}(w)=\phi\left(\frac{\left\|\sqrt{(\epsilon\gamma_l)^2+\sum_{i=1}^p\left((w-s_l)^{(i)}\right)^2}\right\|_{p_n}-\gamma_l\sqrt{1+\epsilon^2}}{\eta_n}\right)\text{,}
\end{equation}
where $\|w\|_{p_n}:=\left(\int_0^1|w(t)|^{p_n}dt\right)^{1/p_n}$ for any $w\in D^p$. We have the following result:
\begin{lemma}\label{lemma413}
For any finite $L$:
\begin{align}\label{app4}
\left\|\prod_{l=1}^L g_{l,n}\right\|_{M^0}\leq \tilde{C}p_n^2\eta_n^{-3}.
\end{align}
for a constant $\tilde{C}$ independent of $p_n$ and $\eta_n$ (which might depend on $\epsilon$ or $\gamma_l$'s).
\end{lemma}
\begin{proof}
First, $\phi$, $\phi'$, $\phi''$, $\phi'''$ are all everywhere continuous and constant outside of the compact interval $[0,1]$ and therefore bounded. Therefore also $\frac{|\phi''(x+h)-\phi''(x)|}{|h|}$ must be uniformly bounded. 

Furthermore, let 
\begin{equation}\label{f_weak}
f(w)=\frac{\left\|\sqrt{(\epsilon\gamma_l)^2+\sum_{i=1}^p\left((w-s_l)^{(i)}\right)^2}\right\|_{p_n}}{\eta_n},
\end{equation}
and denote by $|\cdot|$ the Euclidean norm, and by $\left<\cdot\right>$ the Euclidean inner product.
\\
\begin{center}
\textbf{Step 1: Bounding the first derivative of $f$ of (\ref{f_weak})}
\end{center}

We have that, for any $h\in D^p$,
\begin{align}
Df(w)[h]=&\frac{1}{p_n\eta_n}\left(\int_0^1((\epsilon\gamma_l)^2+|w-s_l|^2(t))^{p_n/2}dt\right)^{1/{p_n}-1}\nonumber\\
&\cdot\frac{p_n}{2}\int_0^1\left((\epsilon\gamma_l)^2+|w-s_l|^2(t)\right)^{p_n/2-1}\cdot 2\left<(w-s_l)(t),h(t)\right>dt.\label{eq_lemma_413}
\end{align}
Applying H{\"o}lder's inequality with coefficients $\frac{p_n}{p_n-2k}$ and $\frac{p_n}{2k}$ and Cauchy-Schwarz inequality, we obtain that, for any $k=1,2,3,$ and $h_1,\dots,h_k\in D^p$,
\begin{align}
&\left|\int_0^1\left((\epsilon\gamma_l)^2+|w-s_l|^2(t)\right)^{p_n/2-k}\left<(w-s_l)(t),h_1(t)\right>\dots\left<(w-s_l)(t),h_k(t)\right>  dt\right|\nonumber\\
\leq&\left(\int_0^1\left((\epsilon\gamma_l)^2+|w-s_l|^2(t)\right)^{p_n/2}dt\right)^{1-2k/p_n}\nonumber\\
&\cdot\left(\int_0^1|w-s_l|^{p_n/2}(t)|h_1|^{p_n/(2k)}(t)\dots |h_k|^{p_n/(2k)}(t)dt\right)^{2k/p_n}\nonumber\\
\leq&\left(\int_0^1\left((\epsilon\gamma_l)^2+|w-s_l|^2(t)\right)^{p_n/2}dt\right)^{1-2k/p_n}\cdot\left(\int_0^1|w-s_l|^{p_n}(t)dt\right)^{k/{p_n}}\prod_{i=1}^k\|h_i\|_{p_n}.\label{eq_lemma_413.1}
\end{align}
Applying (\ref{eq_lemma_413.1}) for $k=1$, together with (\ref{eq_lemma_413}), we get
$$|Df(w)[h]|\leq \frac{1}{\eta_n}\left(\frac{\int_0^1|w-s_l|^{p_n}(t)dt}{\int_0^1((\epsilon\gamma_l)^2+|w-s|^2(t))^{p_n/2}dt}\right)^{1/p_n}\|h\|_{p_n}\leq \frac{\|h\|_{\infty}}{\eta_n}$$
and so
\begin{equation}\label{first}
\sup_{w\in D^p}\|Df(w)\|\leq\frac{1}{\eta_n}.
\end{equation}
\\
\begin{center}
\textbf{Step 2: Bounding the second derivative of $f$ of (\ref{f_weak})}
\end{center}

Note that, for any $h_1,h_2\in D^p$,
\begin{equation}\label{eq_lemma_413.2}
D^2f(w)[h_1,h_2]=A+B
\end{equation}
for
\begin{align}
A=& \frac{1}{\eta_n}\left[\int_0^1\left((\epsilon\gamma_l)^2+|w-s_l|^2(t)\right)^{p_n/2-1}\cdot \left<(w-s_l)(t),h_2(t)\right>dt\right]\nonumber\\
&\cdot\frac{1-p_n}{p_n}\left[\int_0^1\left((\epsilon\gamma_l)^2+|w-s_l|^2(t)\right)^{p_n/2}dt\right]^{1/p_n-2}\nonumber\\
&\cdot \frac{p_n}{2}\int_{0}^1\left((\epsilon\gamma_l)^2+|w-s_l|^2(t)\right)^{p_n/2-1}\cdot 2\left<(w-s_l)(t),h_1(t)\right>dt\nonumber\\
=&\frac{1-p_n}{\eta_n}\prod_{i=1}^2\left\lbrace\left[\int_0^1\left((\epsilon\gamma_l)^2+|w-s_l|^2(t)\right)^{p_n/2-1}\cdot \left<(w-s_l)(t),h_i(t)\right>dt\right]\right\rbrace\nonumber\\
&\cdot \left[\int_0^1\left((\epsilon\gamma_l)^2+|w-s_l|^2(t)\right)^{p_n/2}dt\right]^{1/p_n-2}\nonumber\\
B=&\frac{1}{\eta_n}\left[\int_0^1\left((\epsilon\gamma_l)^2+|w-s_l|^2(t)\right)^{p_n/2}dt\right]^{1/p_n-1}\nonumber\\
&\cdot\left[\int_0^1\frac{p_n-2}{2}\left((\epsilon\gamma_l)^2+|w-s_l|^2(t)\right)^{p_n/2-2}\cdot 2\left<(w-s_l)(t),h_1(t)\right>\left<(w-s_l)(t),h_2(t)\right>dt\right.\nonumber\\
&\left.+\int_0^1\left((\epsilon\gamma_l)^2+|w-s_l|^2(t)\right)^{p_n/2-1}\left<h_1(t),h_2(t)\right>dt\right].\label{ab}
\end{align}
Notice that, by (\ref{eq_lemma_413.1}) with $k=1$,
\begin{align}
|A|\leq& \frac{p_n-1}{\eta_n}\left(\frac{\left(\int_0^1|w-s_l|^{p_n}(t)dt\right)^2}{\left(\int_0^1((\epsilon\gamma_l)^2+|w-s_l|^2(t))^{p_n/2}dt\right)^3}\right)^{1/p_n}\|h_1\|_{p_n}\|h_2\|_{p_n}.\label{eq_leq1}
\end{align}
Furthermore, by H{\"older}'s inequality with coefficients $\frac{p_n}{p_n-2}$ and $\frac{p_n}{2}$ and by the Cauchy-Schwarz inequality,
\begin{align}
&\left|\int_0^1\left((\epsilon\gamma_l)^2+|w-s_l|^2(t)\right)^{p_n/2-1}\left<h_1(t),h_2(t)\right>dt\right|\nonumber\\
\leq& \left(\int_0^1\left((\epsilon\gamma_l)^2+|w-s_l|^2(t)\right)^{p_n/2}\right)^{1-2/p_n}\left(\int_0^1\left<h_1(t),h_2(t)\right>^{p_n/2}\right)^{2/p_n}\nonumber\\
\leq&\left(\int_0^1\left((\epsilon\gamma_l)^2+|w-s_l|^2(t)\right)^{p_n/2}\right)^{1-2/p_n}\|h_1\|_{p_n}\|h_2\|_{p_n}.\label{eq_lemma_413.4}
\end{align}
By (\ref{eq_lemma_413.1}) and (\ref{eq_lemma_413.4}),
\begin{align}
|B|\leq&\frac{p_n-2}{\eta_n}\left(\frac{\left(\int_0^1|w-s_l|^{p_n}(t)dt\right)^2}{\left(\int_0^1((\epsilon\gamma_l)^2+|w-s_l|^2(t))^{p_n/2}dt\right)^3}\right)^{1/p_n}\|h_1\|_{p_n}\|h_2\|_{p_n}\nonumber\\
&+\frac{1}{\eta_n}\left(\int_0^1\left((\epsilon\gamma_l)^2+|w-s_l|^2(t)\right)^{p_n/2}dt \right)^{-1/p_n}\|h_1\|_{p_n}\|h_2\|_{p_n}.\label{eq_leq}
\end{align}
By (\ref{eq_lemma_413.2}), (\ref{eq_leq1}) and (\ref{eq_leq}),
\begin{align*}
&|D^2f(w)[h_1,h_2]|\\
\leq&\left[\frac{2p_n-3}{\eta_n}\left(\frac{\left(\int_0^1|w-s_l|^{p_n}(t)dt\right)^2}{\left(\int_0^1((\epsilon\gamma_l)^2+|w-s_l|^2(t))^{p_n/2}dt\right)^3}\right)^{1/p_n}\right.\\
&+\left.\vphantom{\left(frac{\int_0^t}{\int_0^t}\right)^2}\frac{1}{\eta_n}\left(\int_0^1\left((\epsilon\gamma_l)^2+|w-s_l|^2(t)\right)^{p_n/2}dt \right)^{-1/p_n}\right]\|h_1\|_{p_n}\|h_2\|_{p_n}\\
=&\frac{1}{\eta_n}\left(\int_0^1\left((\epsilon\gamma_l)^2+|w-s_l|^2(t)\right)^{p_n/2}dt \right)^{-1/p_n}\\
&\cdot\left[(2p_n-3)\left(\frac{\left(\int_0^1|w-s_l|^{p_n}(t)dt\right)^2}{\left(\int_0^1((\epsilon\gamma_l)^2+|w-s_l|^2(t))^{p_n/2}dt\right)^2}\right)^{1/p_n}+1\right]\|h_1\|_{p_n}\|h_2\|_{p_n}\\
\leq&\frac{2p_n-2}{\eta_n}\left(\int_0^1\left((\epsilon\gamma_l)^2+|w-s_l|^2(t)\right)^{p_n/2}dt \right)^{-1/p_n}\|h_1\|_{p_n}\|h_2\|_{p_n}\\
\leq&\frac{2p_n-2}{\eta_n(\epsilon\gamma_l)}\|h_1\|_{\infty}\|h_2\|_{\infty}
\end{align*}
and so
\begin{equation}\label{second}
\sup_{w\in D^p}\|D^2f(w)\|\leq 2\frac{p_n-1}{\eta_n(\epsilon\gamma_l)}.
\end{equation}
\\
\begin{center}
\textbf{Step 3: Bounding the third derivative of $f$ of (\ref{f_weak})}
\end{center}

Finally, for any $h_1,h_2,h_3\in D^p$,
\begin{equation}\label{eq_lemma_413.5}
D^3f(w)[h_1,h_2,h_3]=C+D,
\end{equation}
where $C$ comes from differentiating $A$ of (\ref{ab}) and is given by
\begin{align*}
C=E+F
\end{align*}
for
\begin{align}
E=&\frac{1-p_n}{\eta_n}\sum_{1\leq i\neq j\leq 2}\left\lbrace\int_0^1\left((\epsilon\gamma_l)^2+|w-s_l|^2(t)\right)^{p_n/2-1}\left<(w-s_l)(t),h_i(t)\right>dt\right.\nonumber\\
&\phantom{+}\cdot \int_0^1\left[\frac{p_n-2}{2}\left((\epsilon\gamma_l)^2+|w-s_l|^2(t)\right)^{p_n/2-2}\left<(w-s_l)(t),h_j(t)\right>\cdot 2\left<(w-s_l)(t),h_3(t)\right>\right.\nonumber\nonumber\\
&\phantom{+}+\left.\left((\epsilon\gamma_l)^2+|w-s_l|^2(t)\right)^{p_n/2-1}\left<h_j(t),h_3(t)\right>\right]dt\nonumber\nonumber\\
&\left.\phantom{+}\cdot \left[\int_0^1\left((\epsilon\gamma_l)^2+|w-s_l|^2(t)\right)^{p_n/2}dt\right]^{1/p_n-2}\vphantom{\int_0^1}\right\rbrace\nonumber\nonumber\\
F=&\frac{(1-p_n)(1-2p_n)}{p_n\eta_n}\left\lbrace\prod_{i=1}^3\left[\int_0^1\left((\epsilon\gamma_l)^2+|w-s_l|^2(t)\right)^{p_n/2-1}\cdot \left<(w-s_l)(t),h_i(t)\right>dt\right]\right\rbrace\nonumber\\
&\phantom{+}\cdot\left[\int_0^1\left((\epsilon\gamma_l)^2+|w-s_l|^2(t)\right)^{p_n/2}dt\right]^{1/p_n-3}\label{ef}
\end{align}
and $D$ comes from differentiating $B$ of (\ref{ab}) and is given by
\begin{align*}
D=G+H
\end{align*}
for
\begin{align}
G=&\frac{1-p_n}{\eta_n}\left[\int_0^1\left((\epsilon\gamma_l)^2+|w-s_l|^2(t)\right)^{p_n/2}dt\right]^{1/p_n-2}\nonumber\\
&\phantom{+}\cdot\int_0^1\left((\epsilon\gamma_l)^2+|w-s_l|^2(t)\right)^{p_n/2-1}\left<(w-s_l)(t),h_3(t)\right>dt\nonumber\\
&\phantom{+}\cdot\left[\int_0^1(p_n-2)\left((\epsilon\gamma_l)^2+|w-s_l|^2(t)\right)^{p_n/2-2}\cdot \left<(w-s_l)(t),h_1(t)\right>\left<(w-s_l)(t),h_2(t)\right>dt\right.\nonumber\nonumber\\
&\left.\phantom{+}+\int_0^1\left((\epsilon\gamma_l)^2+|w-s_l|^2(t)\right)^{p_n/2-1}\left<h_1(t),h_2(t)\right>dt\right]\nonumber\\
H=&\frac{p_n-2}{\eta_n}\left[\int_0^1\left((\epsilon\gamma_l)^2+|w-s_l|^2(t)\right)^{p_n/2}dt\right]^{1/p_n-1}\nonumber\\
&\phantom{+}\cdot\left\lbrace \int_0^1\left[(p_n-2)\left((\epsilon\gamma_l)^2+|w-s_l|^2(t)\right)^{p_n/2-2}\underset{i,j,k\text{ distinct}}{\sum_{1\leq i,j,k\leq 3}}\left<(w-s_l)(t),h_i(t)\right>\left<h_j(t),h_k(t)\right>\right]dt\right.\nonumber\\
&\phantom{+}+\left.(p_n-4)\int_0^1\left[\left((\epsilon\gamma_l)^2+|w-s_l|^2(t)\right)^{p_n/2-3}\prod_{i=1}^3\left<(w-s_l)(t),h_i(t)\right>\right]dt\vphantom{\underset{l}{\sum_l}}\right\rbrace.\label{gh}
\end{align}
So
\begin{equation}\label{efgh}
D^3f(w)[h_1,h_2,h_3]=E+F+G+H
\end{equation}
for $E,F,G,H$ defined by (\ref{ef}) and (\ref{gh}). By (\ref{eq_lemma_413.1}) and (\ref{eq_lemma_413.4}),
\begin{align}
|E|\leq& \frac{2(p_n-1)\|h_1\|_{p_n}\|h_2\|_{p_n}\|h_3\|_{p_n}}{\eta_n}\nonumber\\
&\cdot\left(\frac{(p_n-2)\left(\int_0^1|w-s_l|^{p_n}(t)dt\right)^{3/p_n}}{\left(\int_0^1\left((\epsilon\gamma_l)^2+|w-s_l|^2(t)\right)^{p_n/2}dt\right)^{5/p_n}}+\frac{\left(\int_0^1|w-s_l|^{p_n}(t)dt\right)^{1/p_n}}{\left(\int_0^1\left((\epsilon\gamma_l)^2+|w-s_l|^2(t)\right)^{p_n/2}dt\right)^{3/p_n}}\right)\nonumber\\
|F|\leq&\frac{(p_n-1)(2p_n-1)\|h_1\|_{p_n}\|h_2\|_{p_n}\|h_3\|_{p_n}}{p_n\eta_n}\cdot\frac{\left(\int_0^1|w-s_l|^{p_n}(t)dt\right)^{3/p_n}}{\left(\int_0^1\left((\epsilon\gamma_l)^2+|w-s_l|^2(t)\right)^{p_n/2}dt\right)^{5/p_n}}\nonumber\\
|G|\leq&\frac{(p_n-1)\|h_1\|_{p_n}\|h_2\|_{p_n}\|h_3\|_{p_n}}{\eta_n}\nonumber\\
&\cdot\left(\frac{(p_n-2)\left(\int_0^1|w-s_l|^{p_n}(t)dt\right)^{3/p_n}}{\left(\int_0^1\left((\epsilon\gamma_l)^2+|w-s_l|^2(t)\right)^{p_n/2}dt\right)^{5/p_n}}+\frac{\left(\int_0^1|w-s_l|^{p_n}(t)dt\right)^{1/p_n}}{\left(\int_0^1\left((\epsilon\gamma_l)^2+|w-s_l|^2(t)\right)^{p_n/2}dt\right)^{3/p_n}}\right)\nonumber\\
|H|\leq &\frac{(p_n-2)\|h_1\|_{p_n}\|h_2\|_{p_n}\|h_3\|_{p_n}}{\eta_n}\nonumber\\
&\cdot\left(\frac{(p_n-4)\left(\int_0^1|w-s_l|^{p_n}(t)dt\right)^{3/p_n}}{\left(\int_0^1\left((\epsilon\gamma_l)^2+|w-s_l|^2(t)\right)^{p_n/2}dt\right)^{5/p_n}}+\frac{6\left(\int_0^1|w-s_l|^{p_n}(t)dt\right)^{1/p_n}}{\left(\int_0^1\left((\epsilon\gamma_l)^2+|w-s_l|^2(t)\right)^{p_n/2}dt\right)^{3/p_n}}\right), \label{efgh_bounds}
\end{align}
where the inequality for $|H|$ uses the following bound obtained by applying H{\"o}lder's inequality with coefficients $\frac{p_n}{p_n-4}$ and $\frac{p_n}{4}$ and Cauchy-Schwarz inequality
\begin{align*}
&\left|\int_0^1\left[\left((\epsilon\gamma_l)^2+|w-s_l|^2(t)\right)^{p_n/2-2}\underset{i,j,k\text{ distinct}}{\sum_{1\leq i,j,k\leq 3}}\left<(w-s_l)(t),h_i(t)\right>\left<h_j(t),h_k(t)\right>\right]dt\right|\\
\leq&\underset{i,j,k\text{ distinct}}{\sum_{1\leq i,j,k\leq 3}}\left(\int_0^1\left((\epsilon\gamma_l)^2+|w-s_l|^2(t)\right)^{p_n/2}dt\right)^{1-4/p_n}\left(\int_0^1|w-s_l|^{p_n/4}(t)\prod_{i=1}^3|h_i|^{p_n/4}(t)dt\right)^{4/p_n}\\
\leq&\underset{i,j,k\text{ distinct}}{\sum_{1\leq i,j,k\leq 3}}\left(\int_0^1\left((\epsilon\gamma_l)^2+|w-s_l|^2(t)\right)^{p_n/2}dt\right)^{1-4/p_n}\left(\int_0^1|w-s_l|^{p_n}(t)dt\right)^{1/p_n}\prod_{i=1}^3\|h_i\|_{p_n}.
\end{align*}
By (\ref{efgh}) and (\ref{efgh_bounds}),
\begin{align*}
&|D^3f(w)[h_1,h_2,h_3]|\\
\leq& \frac{6p_n^2\left(\int_0^1|w-s_l|^{p_n}(t)dt\right)^{3/p_n}\|h_1\|_{p_n}\|h_2\|_{p_n}\|h_3\|_{p_n}}{\eta_n\left(\int_0^1\left((\epsilon\gamma_l)^2+|w-s_l|^2(t)\right)^{p_n/2}dt\right)^{5/p_n}}\\
&+\frac{9p_n\left(\int_0^1|w-s_l|^{p_n}(t)dt\right)^{1/p_n}\|h_1\|_{p_n}\|h_2\|_{p_n}\|h_3\|_{p_n}}{\eta_n\left(\int_0^1\left((\epsilon\gamma_l)^2+|w-s_l|^2(t)\right)^{p_n/2}dt\right)^{3/p_n}}\\
\leq&\frac{15p_n^2\|h_1\|_{p_n}\|h_2\|_{p_n}\|h_3\|_{p_n}}{\eta_n}\left(\int_0^1\left((\epsilon\gamma_l)^2+|w-s_l|^2(t)\right)^{p_n/2}dt\right)^{-2/p_n}\\
\leq&\frac{15p_n^2}{(\epsilon\gamma_l)^2\eta_n}\|h_1\|_{\infty}\|h_2\|_{\infty}\|h_3\|_{\infty}
\end{align*}
and so
\begin{align}\label{third}
\|D^3f(w)\|\leq \frac{15p_n^2}{(\epsilon\gamma_l)^2\eta_n}.
\end{align}
\\
\begin{center}
\textbf{Step 4: Combining the bounds}
\end{center}

The result now follows by combining (\ref{first}), (\ref{second}) and (\ref{third}). Indeed, note that, by the chain rule,
\begin{align*}
&D^3g_{l,n}(w)[h_1,h_2,h_3]\\
=&\phi'''\left(f(w)-\frac{\gamma_l\sqrt{1+\epsilon^2}}{\eta_n}\right)\cdot\prod_{i=1}^3Df(w)[h_i]\\
&+\phi''\left(f(w)-\frac{\gamma_l\sqrt{1+\epsilon^2}}{\eta_n}\right)\cdot\underset{i,j,k \text{distinct}}{\sum_{1\leq i,j,k\leq 3}} D^2f(w)[h_i,h_j] Df(w)[h_k]\\
&+\phi'\left(f(w)-\frac{\gamma_l\sqrt{1+\epsilon^2}}{\eta_n}\right)D^3f(w)[h_1,h_2,h_3].
\end{align*}
By (\ref{first}), (\ref{second}) and (\ref{third}) and the fact that $\phi',\phi'',\phi'''$ are all bounded, we get that, for all $w\in D^p$,
$$\|D^3g_{l,n}(w)\|\leq C_3p_n^2\eta_n^{-3},$$
for some constant $C_3$. Similar bounds may be obtained for the first and second derivative of $g_{l,n}$:
$$\|Dg_{l,n}(w)\|\leq C_1\eta_n^{-1},\quad \|D^2g_{l,n}(w)\|\leq C_2p_n\eta_n^{-1},$$
for constants $C_1,C_2.$ Since $\phi$ is also bounded, the product rule yields the desired bound.
\end{proof}
Now, we prove the following result:
\begin{lemma}\label{lemma_b}
For the set $B$ fixed at the beggining of this Appendix,
$$\limsup_{n\to\infty}\mathbb{P}[\mathbf{Y}_n\in B]\leq \mathbb{P}[\mathbf{Z}\in B]\quad\text{and}\quad \liminf_{n\to\infty}\mathbb{P}[\mathbf{Y}_n\in B]\geq \mathbb{P}[\mathbf{Z}\in B].$$
\end{lemma}
\begin{proof}~
\begin{center}
\textbf{Step 1: Proving the first inequality}
\end{center}
Note that
\begin{align}
&\mathbf{Y}_n\in B_l\Longrightarrow\|\mathbf{Y}_n-s_l\|<\gamma_l\Longrightarrow \sup_{t\in[0,1]}\sum_{i=1}^p\left(\left(\mathbf{Y}_n(t)-s_l(t)\right)^{(i)}\right)^2<\gamma_l^2\nonumber\\
&\Longrightarrow \sup_{t\in[0,1]}\left[\sum_{i=1}^p\left(\left(\mathbf{Y}_n(t)-s_l(t)\right)^{(i)}\right)^2+(\epsilon\gamma_l)^2\right]<\gamma_l^2(1+\epsilon^2)\nonumber\\
&\Longrightarrow\left\|\sqrt{(\epsilon\gamma_l)^2+\sum_{i=1}^p\left((\mathbf{Y}_n-s_l)^{(i)}\right)^2}\right\|_{p_n}\leq \gamma_l\sqrt{1+\epsilon^2}\Longrightarrow g_{l,n}(\mathbf{Y}_n)=1.\label{app0}
\end{align}
Therefore, for all $l$, 
\begin{align}\label{app1}
\mathbb{1}_{[\mathbf{Y}_n\in B_l]}\leq g_{l,n}(\mathbf{Y}_n). 
\end{align}
Also, note that, by Minkowski's inequality and the triangle inequality for the Euclidean norm:
\begin{align*}
&\left\|\sqrt{(\epsilon\gamma_l)^2+\sum_{i=1}^p\left((\mathbf{Z}-s_l)^{(i)}\right)^2}\right\|_{p_n}
\leq\left\|\sqrt{(\epsilon\gamma_l)^2+\sum_{i=1}^p\left((\mathbf{Z}_n-s_l)^{(i)}\right)^2}\right\|_{p_n}+\|\mathbf{Z}_n-\mathbf{Z}\|\text{.}
\end{align*}
Therefore, if $\|\mathbf{Z}-s_l\|>\gamma_l$ then as $p_n\xrightarrow{n\to\infty}\infty$:
\begin{align*}
&\liminf_{n\to\infty}\left\|\sqrt{(\epsilon\gamma_l)^2+\sum_{i=1}^p\left((\mathbf{Z}_n-s_l)^{(i)}\right)^2}\right\|_{p_n}\\
\geq& \liminf_{n\to\infty}\left\lbrace\left\|\sqrt{(\epsilon\gamma_l)^2+\sum_{i=1}^p\left((\mathbf{Z}-s_l)^{(i)}\right)^2}\right\|_{p_n}-\|\mathbf{Z}_n-\mathbf{Z}\|\right\rbrace\\
=&
\sup_{t\in[0,1]}\sqrt{(\epsilon\gamma_l)^2+\sum_{i=1}^p\left((\mathbf{Z}(t)-s_l(t))^{(i)}\right)^2}>\gamma_l(1+\epsilon^2)^{1/2}.
\end{align*}
This, means that, if $p_n\xrightarrow{n\to\infty}\infty$, $\|\mathbf{Z}-s_l\|>\gamma_l$ and $\eta_n\xrightarrow{n\to\infty}0$ then $g_{l,n}(\mathbf{Z}_n)=0$ for sufficiently large $n$, i.e.
\begin{align}\label{app2}
g_{l,n}(\mathbf{Z}_n)\leq\mathbb{1}_{[\|\mathbf{Z}-s_l\|\leq\gamma_l]},\quad \text{as long as } p_n\xrightarrow{n\to\infty}\infty, \eta_n\xrightarrow{n\to\infty}0 \text{ and } n \text{ is large.}
\end{align} 
 By those properties, taking $p_n\to\infty$ and $\eta_n\to 0$ such that $\kappa_n\eta_n^{-3}p_n^2\to 0$, we obtain:
\begin{align*}
\limsup_{n\to\infty}\mathbb{P}[\mathbf{Y}_n\in B]\stackrel{(\ref{app1})}\leq &\limsup_{n\to\infty}\mathbb{E}\left[\prod_{l=1}^Lg_{l,n}(\mathbf{Y}_n)\right]\nonumber\\
\stackrel{(\ref{assumption})}\leq&\limsup_{n\to\infty}\left\lbrace\mathbb{E}\left[\prod_{l=1}^Lg_{l,n}(\mathbf{Z}_n)\right]+C\kappa_n\left\|\prod_{l=1}^Lg_{l,n}\right\|_{M^0}\right\rbrace\nonumber\\
\stackrel{\text{Fatou},(\ref{app4})}{\leq}&\mathbb{E}\left[\limsup_{n\to\infty}\prod_{l=1}^Lg_{l,n}(\mathbf{Z}_n)\right]\stackrel{(\ref{app2})}\leq\mathbb{P}[\mathbf{Z}\in B]\text{.}\label{2.6}
\end{align*}
\begin{center}
\textbf{Step 2: Proving the second inequality}
\end{center}
We define:
\begin{equation}\label{g_ln_star}
g^*_{l,n}(w)=\phi\left(\frac{\left\|\sqrt{(\epsilon\gamma_l)^2+\sum_{i=1}^p\left((w-s_l)^{(i)}\right)^2}\right\|_{p_n}-\gamma_l\sqrt{\epsilon^2+(1-\theta)^2}(\delta\wedge\frac{r_n}{2})^{1/p_n}+\eta_n}{\eta_n}\right)
\end{equation}
for $\theta>0$ fixed and $\delta>0$ such that:
\begin{equation*}
\forall n\in\mathbb{N}:\quad\|\mathbf{Y}_n-s_l\|\geq \gamma_l\Longrightarrow \text{leb}\lbrace t:|\mathbf{Y}_n(t)-s_l(t)|\geq \gamma_l(1-\theta)\rbrace \geq \left(\delta\wedge\frac{r_n}{2}\right)\text{,}
\end{equation*}
where leb denotes the Lebesgue measure.
Such a $\delta$ exists for the following reason. The collection $(s_l,1\leq l\leq L)$ is uniformly equicontinuous and $\mathbf{Y}_n$ are constant on intervals of length at least $r_n$. The $\delta>0$ we choose is such that:
$$|t_1-t_2|\leq \delta\quad\Longrightarrow\quad |s_l(t_1)-s_l(t_2)|\leq \frac{\theta\gamma_l}{2}\text{.}$$
If $\|\mathbf{Y}_n-s_l\|\geq\gamma_l$ then $|\mathbf{Y}_n(t_0)-s_l(t_0)|>\gamma_l\left(1-\frac{\theta}{2}\right)$ for some $t_0$. Then, there exists an interval $I_0$ with $t_0$ being one of its endpoints and of length $\frac{r_n}{2}\wedge \delta$, such that $\mathbf{Y}_n$ is constant on $I_0$ and $|s_l(t)-s_l(t_0)|\leq\frac{\theta\gamma_l}{2}$ for all $t\in I_0$. Then, for $t\in I_0$ we obtain:
\begin{align*}
&|\mathbf{Y}_n(t)-s_l(t)|\geq |\mathbf{Y}_n(t_0)-s_l(t_0)|-|\mathbf{Y}_n(t_0)-\mathbf{Y}_n(t)|-|s_l(t)-s_l(t_0)|\\
\geq&\left(1-\frac{\theta}{2}\right)\gamma_l-\frac{\theta\gamma_l}{2}=\gamma_l(1-\theta)\text{.}
\end{align*}

It follows that:
\begin{align}
&\|\mathbf{Y}_n-s_l\|\geq \gamma_l {\Longrightarrow} \left\|\sqrt{\sum_{i=1}^p\left((\mathbf{Y}_n-s_l)^{(i)}\right)^2}\right\|_{p_n}\geq \gamma_l(1-\theta)\left(\delta\wedge\frac{r_n}{2}\right)^{1/p_n}\Longrightarrow\nonumber\\
&
\left\|\sqrt{(\epsilon\gamma_l)^2+\sum_{i=1}^p\left((\mathbf{Y}_n-s_l)^{(i)}\right)^2}\right\|_{p_n}\geq\gamma_l\sqrt{\epsilon^2+(1-\theta)^2}\left(\delta\wedge\frac{r_n}{2}\right)^{1/p_n}\Longrightarrow g^*_{l,n}(\mathbf{Y}_n)=0\text{.}\label{g_ln_star_in}
\end{align}
Therefore, for all $l$:
\begin{align}\label{app3} 
\mathbb{1}_{[\mathbf{Y}_n\in B_l]}\geq g^*_l(\mathbf{Y}_n). 
\end{align}
Also, again, it can be shown that for any finite $L$ and $\gamma:=\min_{1\leq l\leq L}\gamma_l$:
\begin{align}\label{app5}
\left\|\prod_{l=1}^L g^*_{l,n}\right\|_{M^0}\leq Cp_n^2(\epsilon\gamma)^{-2}\eta_n^{-3}\quad
\text{for some constant }C \text{ independent of }p_n, \epsilon, \gamma \text{ and }\eta_n.
\end{align}
Now suppose $\eta_n\to 0$, $p_n\to \infty$ and $r_n^{1/p_n}\to 1$. Also suppose that $\|\mathbf{Z}-s_l\|<\gamma_l(1-\theta)$ so that there exists $\alpha>0$ such that a.s. $\|\mathbf{Z}_n-s_l\|<\gamma_l(1-\theta)-\alpha$ for $n$ large enough. Then, for large $n$:
\begin{align*}
\left\|\sqrt{(\epsilon\gamma_l)^2+\sum_{i=1}^p\left((\mathbf{Z}_n-s_l)^{(i)}\right)^2}\right\|_{p_n}&\leq
\sqrt{(\epsilon\gamma_l)^2+\|\mathbf{Z}_n-s_l\|^2}\leq\gamma_l\sqrt{\epsilon^2+(1-\theta-\alpha\gamma_l^{-1})^2}\\
&<\gamma_l\sqrt{\epsilon^2+(1-\theta)^2}\left(\delta\wedge\frac{r_n}{2}\right)^{1/p_n}-\eta_n
\end{align*}
because $\left(\delta\wedge\frac{r_n}{2}\right)^{1/p_n}\xrightarrow{n\to\infty}1$ and $\eta_n\xrightarrow{n\to\infty}0$.
So if  $\eta_n\to 0$, $p_n\to \infty$ and $r_n^{1/p_n}\to 1$ then:
$$\|\mathbf{Z}-s_l\|<\gamma_l(1-\theta)\Longrightarrow g^*_{l,n}(\mathbf{Z}_n)=1$$
for $n$ large enough, i.e.:
\begin{align}\label{app6}
\mathbb{1}_{[\|\mathbf{Z}-s_l\|<\gamma_l(1-\theta)]}\leq g^*_{l,n}(\mathbf{Z}_n).
\end{align}
Let $\eta_n\to 0$ and $p_n\to \infty$ be such that $r_n^{1/p_n}\to 1$ and $\kappa_np_n^2\eta_n^{-3}\to 0$. This is possible by the assumption that $\kappa_n\log^2(1/r_n)\to 0$. Indeed, having $r_n^{1/p_n}\to 1$, all we require is that $\log\left(r_n^{1/p_n}\right)\eta_n^3\to 0$ slower than $\kappa_n\log^2(1/r_n)\to 0$, because then:
$$\kappa_np_n^2\eta_n^{-3}=\frac{\kappa_n\left(\log(r_n)\right)^2}{\left(\frac{1}{p_n}\log\left(r_n\right)\right)^2\eta_n^3}=\frac{\kappa_n\left(\log(1/r_n)\right)^2}{\left(\log\left(r_n^{1/p_n}\right)\right)^2\eta_n^3}\to 0$$
For instance, if $r_n\to 0$ and $\kappa_n\to 0$, we require $p_n$ and $\eta_n$ to be such that $\frac{\eta_n^3}{\kappa_n}\to \infty$ and $p_n^2\to\infty$ faster than $(\log r_n)^2$ but slower than $\frac{\eta_n^3}{\kappa_n}$.

Then:
\begin{align*}
\liminf_{n\to\infty}\mathbb{P}\left[\mathbf{Y}_n\in B\right]&\stackrel{(\ref{app3})}\geq\liminf_{n\to\infty}\mathbb{E}\left[\prod_{l=1}^L g^*_{l,n}(\mathbf{Y}_n)\right]\\
&\stackrel{(\ref{assumption})}{\geq}
\liminf_{n\to\infty}\left\lbrace\mathbb{E}\left[\prod_{l=1}^L g^*_{l,n}(\mathbf{Z}_n)\right]-C\kappa_n\left\|\prod_{l=1}^L g^*_{l,n}\right\|_{M^0}\right\rbrace \\
&\stackrel{\text{Fatou},(\ref{app5})}{\geq} \mathbb{E}\left[\liminf_{n\to\infty}\prod_{l=1}^L g^*_{l,n}(\mathbf{Z}_n)\right]\\
&\stackrel{(\ref{app6})}\geq\mathbb{P}\left[\bigcap_{1\leq l\leq L}(\|\mathbf{Z}-s_l\|<\gamma_l(1-\theta))\right]
\text{.}
\end{align*}
Since the choice of $\theta\in(0,1)$ was arbitrary, we conclude that:\\
$\liminf_{n\to\infty}\mathbb{P}\left[\mathbf{Y}_n\in B\right]\geq\mathbb{P}(\mathbf{Z}\in B)\text{.}$
\end{proof}
Lemma \ref{lemma_b} now implies that, for any set $B$ described at the beginning of this Appendix, $\mathbb{P}[\mathbf{Y}_n\in B]\xrightarrow{n\to\infty}\mathbb{P}[\mathbf{\mathbf{Z}}\in B]$, which finishes the proof of Proposition \ref{prop_m}.
\end{appendices}
\section*{Acknowledgements}
The author would like to thank Gesine Reinert and Alison Etheridge for helpful discussions and constructive comments on the early versions of this paper. The author is also grateful to Giovanni Peccati and Christian D{\"o}bler for spotting a mistake in the proof of Lemma \ref{lemma413} and suggesting an alternative approach to proving it.

The author was supported by an EPSRC PhD studentship at the University of Oxford (reference number 1654155) and  the \textbf{FNR grant FoRGES (R-AGR-3376-10)} at the University of Luxembourg.
\bibliographystyle{alpha}
\bibliography{Bibliography}
\end{document}